\documentclass[11pt]{article}
\usepackage{amssymb}
\usepackage{xcolor}

\usepackage{mathrsfs}
\makeatletter
       \@addtoreset{equation}{section}\makeatother
\usepackage{amsfonts}
\newtheorem{Theorem}{Theorem}[section]
\newtheorem{Proposition}[Theorem]{Proposition}

\newtheorem{Lemma}[Theorem]{Lemma}

\newtheorem{Definition}[Theorem]{Definition}

\def\supp{{\rm supp}\ }

\def\dist{{\rm dist}\ }
\def\bbZ{\mathbb{Z}}
\def\bbR{\mathbb{R}}

\begin{document}

\title{Weak and Strong Type Weighted Estimates for Multilinear Calder\'{o}n-Zygmund Operators\thanks{This work was supported partially by the
National Natural Science Foundation of China(10990012) and the Research Fund for the Doctoral Program
of Higher Education.}}

\author{Kangwei Li \quad and Wenchang Sun\thanks{Corresponding author.}\\
School of Mathematical Sciences and LPMC,  Nankai University,
      Tianjin~300071, China\\
Email:\, likangwei9@mail.nankai.edu.cn, sunwch@nankai.edu.cn}
\date{}
\maketitle
\begin{abstract}
In this paper, we study the weighted estimates for multilinear Calder\'{o}n-Zygmund operators 
from $L^{p_1}(w_1)\times\cdots\times L^{p_m}(w_m)$
to $L^{p}(v_{\vec{w}})$, where $1<p, p_1,\cdots,p_m<\infty$ with $1/{p_1}+\cdots+1/{p_m}=1/p$ and $\vec{w}=(w_1,\cdots,w_m)$ is a multiple $A_{\vec{P}}$ weight.
We give weak and strong type weighted estimates of mixed $A_p$-$A_\infty$ type.
Moreover, the strong type weighted estimate is sharp
 whenever
$\max_i p_i \le p'/(mp-1)$.
\end{abstract}


\section{Introduction and Main Results}
The weighted estimate for operators is an interesting topic in harmonic analysis.
And it has attracted many authors in this area \cite{DGGLY,Hu-Yang2009,LOPTT,Lin-Meng-Yang2010,Shi-Fu-Lu2011,Xue-Ding2009}.
In this paper, we study the weighted estimates for multilinear Calder\'{o}n-Zygmund operators with multiple $A_{\vec{P}}$ weights.

Recall that $T$ is called a multilinear Calder\'{o}n-Zygmund operator if
$T$ is initially defined on the $m$-fold product of Schwartz spaces and
taking values into the space of tempered distributions,
\[
  T: \mathscr{S}(\bbR^n)\times\cdots\times \mathscr{S}(\bbR^n)\rightarrow \mathscr{S}'(\bbR^n),
\]
and for some $1\le q_i<\infty$, it extends to a bounded multilinear operator
from $L^{q_1} \times\cdots\times L^{q_m}$ to $L^q$ , where $1/{q_1}+\cdots+1/{q_m}=1/q$,
 and if there exists a function $K$, defined off the diagonal $x=y_1=\cdots=y_m$
in $(\bbR^n)^{m+1}$, satisfying
\[
  T(f_1,\cdots,f_m)(x)=\int_{(\bbR^n)^m}K(x,y_1,\cdots,y_m)f_1(y_1)\cdots f_m(y_m)dy_1\cdots dy_m
\]
for all $x\notin \bigcap_{j=1}^m\supp f_i$;
\[
  |K(y_0,y_1,\cdots,y_m)|\le \frac{A}{(\sum_{k,l=0}^m |y_k-y_l|)^{mn}}
\]
and
\[
  |K(y_0,\cdots,y_i,\cdots,y_m)-K(y_0,\cdots,y_i',\cdots,y_m)|\le \frac{A|y_i-y_i'|^\varepsilon}{(\sum_{k,l=0}^m |y_k-y_l|)^{mn+\varepsilon}}
\]
for some $A,\varepsilon>0$ and all $0\le i\le m$, whenever $|y_i-y_i'|\le \frac{1}{2}\max_{0\le k\le m}|y_i-y_k|$.

For the theory of multilinear Calder\'on-Zygmund operators, we refer the readers to \cite{DGGLY,DGY,Gra,GT,GT2} for an overview.

The multiple $A_{\vec{P}}$ weights introduced by  Lerner, Ombrosi, P\'{e}rez, Torres and Trujillo-Gonz\'{a}lez \cite{LOPTT} are defined as follows.
Let $\vec{P}=(p_1,\cdots,p_m)$ with $1\le p_1,\cdots,p_m<\infty$ and $1/{p_1}+\cdots+1/{p_m}=1/p$. Given $\vec{w}=(w_1,\cdots, w_m)$, set
\[
  v_{\vec{w}}=\prod_{i=1}^m w_i^{p/{p_i}}.
\]
We say that $\vec{w}$ satisfies the multilinear $A_{\vec{P}}$ condition if
\[
  [\vec{w}]_{A_{\vec{P}}}:=\sup_Q \left(\frac{1}{|Q|}\int_Q v_{\vec{w}}\right)\prod_{i=1}^m\left( \frac{1}{|Q|}\int_Q w_i^{1-p_i'}\right)^{p/{p_i'}}<\infty,
\]
where $[\vec{w}]_{A_{\vec{P}}}$ is called the $A_{\vec{P}}$ constant of $\vec{w}$. When $p_i=1$, $( \frac{1}{|Q|}\int_Q w_i^{1-p_i'})^{1/{p_i'}}$ is
understood as $(\inf_Q w_i)^{-1}$. It is easy to see that in the linear case (that is,  $m = 1$) $[\vec{w}]_{A_{\vec{P}}}=[w]_{A_p}$ is the
usual $A_p$ constant. Recall that $A_\infty=\bigcup_{1\le p<\infty}A_p$
and the $A_\infty$ constant $[w]_{A_\infty}$ is defined by
\[
  [w]_{A_\infty}:=\sup_Q \frac{1}{w(Q)}\int_Q M(w\chi_Q).
\]
In \cite{LOPTT}, it was shown that for $1<p_1,\cdots,p_m<\infty$, $\vec{w}\in A_{\vec{P}}$ if and only if $w_i^{1-p_i'}\in A_{mp_i'}$ and $v_{\vec{w}}\in A_{mp}$.

For the linear case, i.e. $m=1$, the $A_p$-$A_\infty$ type estimates for Calder\'{o}n-Zygmund operators were investigated in \cite{HL}.
Notice that the main technique in \cite{HL} is an appropriate characterization which
simplifies the estimate of the weighted bounds to calculate a test condition \cite{HLMORSU,LSU}.
The advantage of their technique is that it does not rely upon the extrapolation.  In this paper,
roughly speaking, we follow the idea used in \cite{HLMORSU}. But we do not use the method such as the linearization used in that paper.
Instead, we  use the idea of Dami\'an, Lerner and P\'erez \cite{DLP} and reduce the problem to consider the following type of operators,
\[
  A^{}_{\mathscr{D},\mathcal{S}}(\vec{f})=\sum_{j,k}\bigg( \prod_{i=1}^m \frac{1}{|Q_{j,k}|}\int_{Q_{j,k}}f_i(y_i)dy_i\bigg)\chi^{}_{Q_{j,k}},
\]
where $\vec{f}:=(f_1,\cdots, f_m)$, $\mathscr{D}$ is a dyadic grid and $\mathcal{S}:=\{Q_{j,k}\}$ is a sparse family in $\mathscr{D}$ (see Section~\ref{sec:s1}
for definitions of these notations).

In the linear case, Lerner \cite{L1,L2} investigated this type of operators and gave a simple proof for the $A_2$ conjecture.
For the fundamental theory of $A_p$ weights and the history of the $A_2$ conjecture, we refer the readers to \cite{B,Cruz-Uribe,H,HP,LPR,P1,P2}
for an overview.

In \cite{DLP}, Dami\'an, Lerner and P\'erez studied the sharp weighted bound of multilinear maximal function of mixed $A_p$-$A_\infty$ type and gave
a multilinear version of the $A_2$ conjecture. In \cite{LS}, the authors estimated the weighted bound of the multilinear maximal function
and Calder\'on-Zygmund operators in terms of $[\vec{w}]_{A_{\vec{P}}}$.

In this paper, we estimate the weighted bound of multilinear Calder\'on-Zygmund operators of mixed $A_p$-$A_\infty$ type.
We give the sharp estimate for some cases. To be
precise, the main result of this paper is the following.

\begin{Theorem}\label{thm:main}
Let $T$ be a multilinear Calder\'{o}n-Zygmund operator, $\vec{P}=(p_1,\cdots,p_m)$ with $1/p=1/{p_1}+\cdots+1/{p_m}$
and $1<p,p_1,\cdots,p_m<\infty$. Suppose that $\vec{w}=(w_1,\cdots,w_m)$ with $\vec{w}\in A_{\vec{P}}$.
Then
\begin{eqnarray}
\|T(\vec{f})\|_{L^p(v_{\vec{w}})}
&\le& C_{m,n,\vec{P},T}[\vec{w}]_{A_{\vec{P}}}^{1/p}\bigg(\prod_{i=1}^m [\sigma_i]_{A_\infty}^{1/{p_i}}
  \nonumber \\
&&\qquad
+ [v_{\vec{w}}]_{A_\infty}^{1/{p'}}\sum_{i'=1}^m \prod_{i\neq i'} [\sigma_i]_{A_\infty}^{1/{p_i}} \bigg)
\prod_{i=1}^m \|f_i\|_{L^{p_i}(w_i)}, \label{eq:main:1}
\end{eqnarray}
where $\sigma_i=w_i^{1-p_i'}$, $i=1,\cdots,m$.
The result is sharp in the sense that the exponents can not be improved whenever
$\max_i p_i \le p'/(mp-1)$.
\end{Theorem}

For the weak type estimates, we get a similar result.
\begin{Theorem}\label{thm:m1}
Let $T$ be a multilinear Calder\'{o}n-Zygmund operator, $\vec{P}=(p_1,\cdots,p_m)$ with $1/p=1/{p_1}+\cdots+1/{p_m}$
and $1<p,p_1,\cdots,p_m<\infty$. Suppose that $\vec{w}:=(w_1,\cdots,w_m)\in A_{\vec{P}}$.
Then we have
\[
\|T(\vec{f})\|_{L^{p,\infty}(v_{\vec{w}})} \le C_{m,n,\vec{P},T}[\vec{w}]_{A_{\vec{P}}}^{1/p}[v_{\vec{w}}]_{A_\infty}^{1/{p'}}
 \bigg(\sum_{i'=1}^m \prod_{i\neq i'} [\sigma_i]_{A_\infty}^{1/{p_i}} \bigg)
\prod_{i=1}^m \|f_i\|_{L^{p_i}(w_i)}.
\]
\end{Theorem}

In the rest of this paper, we give proofs for the main results.
To avoid cumbersome notations, we only prove Theorems~\ref{thm:main} and \ref{thm:m1} for the case $m=2$. And the general case can be proved similarly but with more complicated symbols.

\section{Preliminaries}\label{sec:s1}
In this section, we collect some notations and preliminary results.
Recall that the standard dyadic grid in $\bbR^n$ consists of the cubes
\[
  [0,2^{-k})^n+2^{-k}j,\qquad k\in\bbZ, j\in\bbZ^n.
\]
Denote the standard dyadic grid by $\mathcal{D}$.

By a general dyadic grid $\mathscr{D}$ we mean a collection of cubes with the following
properties: (i) for any $Q\in\mathscr{D}$ its sidelength $l_Q$ is of the form $2^k$, $k\in\bbZ$; (ii)
 $Q\cap R \in \{Q,R,\emptyset\}$ for any $Q,R\in\mathscr{D}$; (iii) the cubes of a fixed sidelength $2^k$ form a partition
of $\bbR^n$.

We say that  $\mathcal{S}:=\{Q_{j,k}\}$ is a sparse family of cubes if:
\begin{enumerate}
\item for each fixed $k$ the cubes $Q_{j,k}$ are pairwise disjoint;
\item if $\Gamma_k=\bigcup_j Q_{j,k}$, then $\Gamma_{k+1}\subset \Gamma_k$;
\item $|\Gamma_{k+1}\bigcap Q_{j,k}|\le \frac{1}{2}|Q_{j,k}|$.
\end{enumerate}
For any $Q_{j,k}\in\mathcal{S}$, we define $E(Q_{j,k})=Q_{j,k}\setminus \Gamma_{k+1}$.
Then the sets $E(Q_{j,k})$ are pairwise disjoint and $|E(Q_{j,k})|\ge \frac{1}{2}|Q_{j,k}|$.

In \cite{DLP}, Dami\'an, Lerner and P\'erez proved that for any Banach function space $\mathcal{X}$ over $\bbR^n$ equipped with Lebesgue measure,
\begin{equation}\label{eq:x}
 \|T(\vec{f})\|_{\mathcal{X}}\le C\sup_{\mathscr{D},\mathcal{S}}\|A_{\mathscr{D},\mathcal{S}}(|\vec{f}|)\|_{\mathcal{X}},
\end{equation}
where $|\vec{f}|=(|f_1|,\cdots, |f_m|)$
and the supremum is taken over arbitrary dyadic grids $\mathscr{D}$ and sparse
families $\mathcal S \subset \mathscr D$.
Specially, for $\mathcal{X}=L^p(v_{\vec{w}})$, $1\le p<\infty$,
\begin{equation}\label{eq:tlesa}
\|T(\vec{f})\|_{L^p(v_{\vec{w}})}\le C\sup_{\mathscr{D},\mathcal{S}}\|A_{\mathscr{D},\mathcal{S}}(|\vec{f}|)\|_{L^p(v_{\vec{w}})}.
\end{equation}

Let $\mathbb{E}_{Q}^\sigma f:=\sigma(Q)^{-1}\int_Q f\sigma$. We introduce the principal cubes \cite{HLMORSU}.
\begin{Definition}[Principal cubes]\label{def:d1}
We form the collection $\mathcal{G}$ of principal cubes as follows.
Let $\mathcal{G}_0:=\{ \overline{Q}\}$ (the maximal dyadic cube that we consider). And inductively,
\[
  \mathcal{G}_k:=\bigcup_{G\in \mathcal{G}_{k-1}}\{G'\subset G: \mathbb{E}_{G'}^\sigma |f|>4 \mathbb{E}_{G}^\sigma |f|, G' \mbox{is a maximal such dyadic cube}\}.
\]
Let $\mathcal{G}:=\bigcup_{k=0}^\infty \mathcal{G}_k$.
For any dyadic $Q(\subset \overline{Q})$, we let
\[
  \Gamma(Q):=\mbox{the minimal principal cube containing $Q$}.
\]
\end{Definition}
It follows from the definition that
\[
  \mathbb{E}_{Q}^\sigma |f|\le 4\mathbb{E}_{\Gamma(Q)}^\sigma |f|.
\]

From the idea of principal cubes, we have the following decomposition, which is similar to the ordinary corona decomposition (See \cite{LPR, PTV}).

Let $\mathcal{Q}\subset\mathscr{D}$ be any collection of dyadic cubes such that for any $Q\in\mathcal{Q}$, there exists a maximal cube $Q_{max}\in \mathcal{Q}$
which contains $Q$.
Let $\sigma_1dx$ and $\sigma_2dx$ be two positive measures.
We call $(\mathcal{L}: \mathcal{Q}(L)): \mathcal{L}\subset \mathcal{Q}$
a ($\sigma_1$, $\sigma_2$)-corona decomposition of $\mathcal{Q}$ if these conditions hold.
\begin{enumerate}
\item For each $Q\in\mathcal{Q}$ there is a member of $\mathcal{L}$ that contains $Q$. Let $\lambda(Q)\in \mathcal{L}$
denote the minimal cube which contains $Q$. Then we have
\[
  4\frac{\sigma_1(\lambda(Q))\sigma_2(\lambda(Q))}{|\lambda(Q)|^2}\ge \frac{\sigma_1(Q)\sigma_2(Q)}{|Q|^2}.
\]
\item For all $L',L\in \mathcal{L}$ with $L'\subsetneq L$,
\[
  \frac{\sigma_1(L')\sigma_2(L')}{|L'|^2}> 4\frac{\sigma_1(L)\sigma_2(L)}{|L|^2}.
\]
\end{enumerate}
We set $\mathcal{Q}(L):=\{Q\in\mathcal{Q}: \lambda(Q)=L\}$.
The collection  $\mathcal{Q}(L)$ forms a partition of  $\mathcal{Q}$.

Note that $\{Q\times Q:\, Q\in\mathcal Q\}$ is a collection of dyadic cubes in
  $\bbR^{2n}$. Therefore, the ($\sigma_1$, $\sigma_2$)-corona decomposition of $\mathcal Q$
   is in fact the ordinary corona decomposition of $\{Q\times Q:\, Q\in\mathcal Q\}$
   with respect to the measure $\sigma_1\times \sigma_2$.

Now we introduce some preliminary results.
The following result is obvious and we omit the proof.
\begin{Lemma}\label{lm:sparse}
Any sub-family of a sparse family is also sparse.
\end{Lemma}

Next we give a property of $A_\infty$ weights on sparse family.
\begin{Lemma}\label{lm:l0}
Let $w\in A_\infty$ and $\mathcal{Q}\subset \mathscr D$ be a sparse family. Suppose that there is
some $S\in\mathscr D$ such that
 any cube  in $\mathcal{Q}$ is contained in $S$.
 Then
\[
  \sum_{Q\in\mathcal{Q}} w(Q)\le 2\int_S M(w 1_S)(x)dx \le 2 [w]_{A_\infty}w(S).
\]
\end{Lemma}
\begin{proof}
Set
\[
  E(Q)=Q\setminus \bigcup_{Q'\in\mathcal Q, Q'\subsetneq Q} Q'.
\]
By the sparse property, $E(Q)$ are disjoint and $|E(Q)|\ge \frac{1}{2}|Q|$.
Then we have
\begin{eqnarray*}
\sum_{Q\in\mathcal{Q}} w(Q)\le 2\sum_{Q\in\mathcal{Q}} \frac{w(Q)}{|Q|}|E(Q)|\le 2\int_S M(w 1_S)(x)dx \le 2 [w]_{A_\infty}w(S).
\end{eqnarray*}
\end{proof}

By (\ref{eq:tlesa}), we have to estimate
$\|A_{\mathscr{D},\mathcal{S}}(|\vec{f}|)\|_{L^p(v_{\vec{w}})}$. First, we consider a special case.

\begin{Lemma}\label{lm:l1}
Suppose that $(w_1,w_2)\in A_{\vec{P}}$, where $\vec{P}=(p_1,p_2)$  and $1/p=1/{p_1}+1/{p_2}$. Let $\mathcal{Q}\subset \mathscr{D}$ be
a sparse family. Suppose that there is some $S\in\mathscr D$ such that
 any cube  in $\mathcal{Q}$ is contained in $S$. Set
\[
  A^{}_{\mathscr{D},\mathcal{Q}}(\vec{f})=\sum_{Q\in\mathcal{Q}}\bigg( \prod_{i=1}^2\frac{1}{|Q|}\int_{Q}f_i(y_i)dy_i\bigg)\chi^{}_{Q}.
\]
Then we have
\begin{eqnarray*}
&&\left\|A_{\mathscr{D},\mathcal{Q}}(\sigma_1 1_S, \sigma_2 1_S)\right\|_{L^p(v_{\vec{w}})}\\
&&\quad\lesssim [\vec{w}]_{A_{\vec{P}}}^{1/p}
 \bigg(\sum_{Q\in\mathcal{Q}}\sigma_1(Q)\bigg)^{1/{p_1}}\cdot\bigg(\sum_{Q\in\mathcal{Q}}\sigma_2(Q)\bigg)^{1/{p_2}}.
\end{eqnarray*}
\end{Lemma}

\begin{proof}
Without loss of generality, assume that $p_1\le p_2$.  Set
\[
   \mathcal{Q}_a:=\bigg\{Q\in\mathcal{Q}: 2^{a}<\bigg(\frac{v_{\vec{w}}(Q)}{|Q|}\bigg)^{1/p}
  \prod_{i=1}^2\bigg(\frac{\sigma_i(Q)}{|Q|}\bigg)^{1/{p_i'}}\le 2^{a+1}\bigg\},
\]
where $-1\le a\le \lfloor\log_2 [\vec{w}]_{A_{\vec{P}}}^{1/p}\rfloor$.
Form the ($\sigma_1$, $\sigma_2$)-corona decomposition of $\mathcal{Q}_a$.
We get $\mathcal{L}_a$. Define
\[
  A^{}_{\mathscr{D},\mathcal{Q}_a(L)}(x)=\sum_{Q\in\mathcal{Q}_a(L)}\frac{\sigma_1(Q)\sigma_2(Q)}{|Q|^2}\chi^{}_{Q}(x).
\]
We conclude that there exists some $c>0$ such that for $L\in\mathcal{L}_a$ and $t\ge0$, we have
\begin{eqnarray}
&& \left|\bigg\{ x\in \bbR^n: \mathcal{A}_{\mathscr{D},\mathcal{Q}_a(L)}(x)
   > 4t\frac{\sigma_1(L)\sigma_2(L)}{|L|^2}\bigg\}\right|\le 2^{-t+2}|L|,\label{eq:e5}\\
&& v_{\vec{w}}\left(\bigg\{ x\in \bbR^n: \mathcal{A}_{\mathscr{D},\mathcal{Q}_a(L)}(x)
   > t\frac{\sigma_1(L)\sigma_2(L)}{|L|^2}\bigg\}\right)\lesssim 2^{-ct}  v_{\vec{w}}(L).\label{eq:e6}
\end{eqnarray}

First, we prove $(\ref{eq:e5})$. It is obvious that
\[
  \left|\bigg\{ x\in \bbR^n: \mathcal{A}_{\mathscr{D},\mathcal{Q}_a(L)}(x)
   > 4\frac{\sigma_1(L)\sigma_2(L)}{|L|^2}\bigg\}\right|\le |L|.
\]
Since for $Q\in\mathcal{Q}_a(L)$,
\[
  \frac{\sigma_1(Q)\sigma_2(Q)}{|Q|^2}\le 4\frac{\sigma_1(L)\sigma_2(L)}{|L|^2},
\]
by the sparse property of $\mathcal{Q}$, we have for any integer $\tau$,
\[
  \left|\bigg\{ x\in \bbR^n: \mathcal{A}_{\mathscr{D},\mathcal{Q}_a(L)}(x)
   > 4\tau\frac{\sigma_1(L)\sigma_2(L)}{|L|^2}\bigg\}\right|\le 2^{-\tau+1}|L|.
\]
This proves $(\ref{eq:e5})$.

Next we prove $(\ref{eq:e6})$. For integers $b\ge 0$, we define
$\mathcal{Q}_{a,b}(L)$ to be the set consisting of  $Q\in\mathcal{Q}_a(L)$ such that
\[
2^{-b+1}\frac{\sigma_1(L)\sigma_2(L)}{|L|^2}< \frac{\sigma_1(Q)\sigma_2(Q)}{|Q|^2}\le 2^{-b+2}\frac{\sigma_1(L)\sigma_2(L)}{|L|^2}.
\]
Define
\[
  E_b(t):=\bigg\{ x\in \bbR^n: \mathcal{A}_{\mathscr{D},
  \mathcal{Q}_{a,b}(L)}(x)>4t2^{-b}\frac{\sigma_1(L)\sigma_2(L)}{|L|^2}\bigg\},
\]
where
\[
  \mathcal{A}_{\mathscr{D},\mathcal{Q}_{a,b}(L)}(x)
  =\sum_{Q\in \mathcal{Q}_{a,b}(L)}
    \frac{\sigma_1(Q)\sigma_2(Q)}{|Q|^2}\chi^{}_{Q}(x).
\]
Similar arguments as the above show that $|E_b(t)|\le  2^{-t+2}|L|$. Let $K=4\sum_{b\ge 0} 2^{-b/2}$.
We have
\begin{eqnarray*}
&& v_{\vec{w}}\left(\bigg\{ x\in \bbR^n: \mathcal{A}_{\mathscr{D},\mathcal{Q}_a(L)}(x)>t\frac{\sigma_1(L)\sigma_2(L)}{|L|^2}\bigg\}\right)\\
&
  \le&
  \sum_{b\ge 0} v_{\vec{w}}\left(\bigg\{ x\in \bbR^n: \mathcal{A}_{\mathscr{D},\mathcal{Q}_{a,b}(L)}(x)>4t2^{-b/2}K^{-1}\frac{\sigma_1(L)\sigma_2(L)}{|L|^2}\bigg\}\right)\\
  &=&\sum_{b\ge 0} v_{\vec{w}}(E_b(2^{b/2}K^{-1} t)).
\end{eqnarray*}
Suppose that $E_b(2^{b/2}K^{-1}t)=\bigcup_j R_j^b$, where $R_j^b$ are pairwise
disjoint maximal dyadic cubes in $E_b(2^{b/2}K^{-1}t)$. Notice that $R_j^b\in\mathcal{Q}_{a,b}(L)$. We have
\begin{eqnarray*}
v_{\vec{w}}(E_b(2^{b/2}K^{-1}t))&=&\sum_j v_{\vec{w}}(R_j^b)\\
&\lesssim& \sum_j \bigg(2^a \frac{|R_j^b|^2}{\sigma_1(R_j^b)^{1/{p_1'}}\sigma_2(R_j^b)^{1/{p_2'}}} \bigg)^p\\
&=& \sum_j \frac{2^{ap}|R_j^b|^{2p/{p_2'}}}{\sigma_1(R_j^b)^{p/{p_2'}}\sigma_2(R_j^b)^{p/{p_2'}}}\sigma_1(R_j^b)^{p/{p_1}-p/{p_2}}|R_j^b|^{2p/{p_2}}\\
&\lesssim&
  \sum_j \frac{2^{ap}2^{bp/{p_2'}}
  |L|^{2p/{p_2'}}}{\sigma_1(L)^{p/{p_2'}}\sigma_2(L)^{p/{p_2'}}}
    \sigma_1(R_j^b)^{p/{p_1}-p/{p_2}}|R_j^b|^{2p/{p_2}}\\
&\le & 2^{bp/{p_2'}}\sum_j \bigg(\frac{\sigma_1(R_j^b)}{\sigma_1(L)}\bigg)^{p/{p_1}-p/{p_2}}\cdot\bigg(\frac{|R_j^b|}{|L|}\bigg)^{2p/{p_2}}v_{\vec{w}}(L)\\
&&\hskip 40mm \mbox{(By the definition of $\mathcal{Q}_a$)}\\
&\le& 2^{bp/{p_2'}}\bigg(\sum_j\frac{\sigma_1(R_j^b)}{\sigma_1(L)}\bigg)^{p/{p_1}-p/{p_2}}
  \!\cdot\!\bigg(\sum_j\frac{|R_j^b|}{|L|}\bigg)^{2p/{p_2}}
   \!\! v_{\vec{w}}(L)\\
&\le& 2^{bp/{p_2'}} 2^{(-K^{-1}2^{b/2}t+2)\frac{2p}{p_2}} v_{\vec{w}}(L).
\end{eqnarray*}
It follows that for $t\ge 1$,
\begin{eqnarray*}
&&v_{\vec{w}}\left(\bigg\{ x\in \bbR^n: \mathcal{A}_{\mathscr{D},\mathcal{Q}_a(L)}(x)>t\frac{\sigma_1(L)\sigma_2(L)}{|L|^2}\bigg\}\right)\\
&\le& \sum_{b\ge 0} v_{\vec{w}}(E_b(2^{b/2}K^{-1}t))\\
&\lesssim& \sum_{b\ge 0} 2^{bp/{p_2'}} 2^{(-K^{-1}2^{b/2}t+2)\frac{2p}{p_2}} v_{\vec{w}}(L)\\
&\le& \sum_{b\ge 0} 2^{bp/{p_2'}} 2^{(-K^{-1}2^{b/2}(1+t)/2+2)\frac{2p}{p_2}} v_{\vec{w}}(L)\\
&\lesssim& 2^{-ct}  v_{\vec{w}}(L),
\end{eqnarray*}
where $c=K^{-1}p/{p_2}$.
For $0\le t<1$, it is obvious that $(\ref{eq:e6})$ is correct.

For $L\in\mathcal{L}_a$ and $d\in\bbZ_+$, let
\[
  L_{a,d}=\Big\{x\in\bbR^n: \mathcal{A}_{\mathscr{D},\mathcal{Q}_a(L)}(x)\in (d,d+1]\frac{\sigma_1(L)\sigma_2(L)}{|L|^2}\Big\}.
\]
It is obvious that $L_{a,d}\subset L$ and by (\ref{eq:e6}),
\[
  v_{\vec{w}}(L_{a,d})\lesssim 2^{-cd}v_{\vec{w}}(L).
\]
By the definition of ($\sigma_1$, $\sigma_2$)-corona decomposition, we have
\[
  \sum_{L\in\mathcal{L}_a}\frac{\sigma_1(L)\sigma_2(L)}{|L|^2}\chi^{}_{L_{a,d}}(x)\asymp
  \bigg(\sum_{L\in\mathcal{L}_a}\bigg(\frac{\sigma_1(L)\sigma_2(L)}{|L|^2}\bigg)^p \chi^{}_{L_{a,d}}(x)\bigg)^{1/p}.
\]
It follows that
\begin{eqnarray*}
&&\left\|A_{\mathscr{D},\mathcal{Q}}(\sigma_1 1_S, \sigma_2 1_S)\right\|_{L^p(v_{\vec{w}})}\\
&\le& \sum_{a=-1}^{\lfloor\log_2 [\vec{w}]_{A_{\vec{P}}}^{1/p}\rfloor}\left\|A_{\mathscr{D},\mathcal{Q}_a}(x)\right\|_{L^p(v_{\vec{w}})}\\
&\le& \sum_{a=-1}^{\lfloor\log_2 [\vec{w}]_{A_{\vec{P}}}^{1/p}\rfloor}\sum_{d= 0}^\infty(d+1)\left\|\sum_{L\in\mathcal{L}_a}\frac{\sigma_1(L)\sigma_2(L)}{|L|^2}\chi^{}_{L_{a,d}}(x)\right\|_{L^p(v_{\vec{w}})}\\
&\lesssim& \sum_{a=-1}^{\lfloor\log_2 [\vec{w}]_{A_{\vec{P}}}^{1/p}\rfloor}\sum_{d= 0}^\infty\frac{d+1}{2^{cd/p}}
\left(\sum_{L\in\mathcal{L}_a}\bigg(\frac{\sigma_1(L)\sigma_2(L)}{|L|^2}\bigg)^p v_{\vec{w}}(L)\right)^{1/p}\\
&\lesssim& \sum_{a=-1}^{\lfloor\log_2 [\vec{w}]_{A_{\vec{P}}}^{1/p}\rfloor} 2^a
\left(\sum_{L\in\mathcal{L}_a} \sigma_1(L)^{p/{p_1}}\sigma_2(L)^{p/{p_2}}\right)^{1/p}\\
&\lesssim &[\vec{w}]_{A_{\vec{P}}}^{1/p}
 \bigg(\sum_{Q\in\mathcal{Q}}\sigma_1(Q)\bigg)^{1/{p_1}}\cdot\bigg(\sum_{Q\in\mathcal{Q}}\sigma_2(Q)\bigg)^{1/{p_2}}.
\end{eqnarray*}
This completes the proof.
\end{proof}

The following result  gives another special case of
$\|A_{\mathscr{D},\mathcal{S}}(|\!\vec{f}|)\|_{L^p(v_{\vec{w}})}$.
Since its proof shares some common steps with the one for Theorem~\ref{thm:main},
we postpone the proof to Section~\ref{sec:s3}.

\begin{Lemma}\label{lm:l2}
Suppose that $(w_1,w_2)\in A_{\vec{P}}$ with $\vec{P}=(p_1,p_2)$ satisfies that $1/p=1/{p_1}+1/{p_2}$ and that $1<p,p_1,p_2<\infty$. Let $S$ be a dyadic cube
and $\supp f_1\subset S$. Then
\begin{eqnarray*}
\left\|1_S A_{\mathscr{D},\mathcal{S}}(|f_1|\sigma_1 , \sigma_2 1_S)\right\|_{L^p(v_{\vec{w}})}&\lesssim &[\vec{w}]_{A_{\vec{P}}}^{1/p}[\sigma_2]_{A_\infty}^{1/{p_2}}
  ([v_{\vec{w}}]_{A_\infty}^{1/{p'}}+[\sigma_1]_{A_\infty}^{1/{p_1}})\\
  &&\quad\cdot \|f_1\|_{L^{p_1}(\sigma_1)}\sigma_2(S)^{1/{p_2}}.
\end{eqnarray*}
\end{Lemma}

To prove the main result, we also need the following result on multiple weights.
\begin{Lemma}\cite[Lemma 2.2]{LS}\label{lm:w}
Suppose that $\vec{w}=(w_1,\cdots,w_m)\in A_{\vec{P}}$ and that $1<p$, $p_1$, $\cdots$, $p_m<\infty$ with $1/{p_1}+\cdots+1/{p_m}=1/p$.
Then $\vec{w}^i:=(w_1$, $\cdots$, $w_{i-1}$, $v_{\vec{w}}^{1-p'}$, $w_{i+1}$, $\cdots$, $w_m)\in A_{\vec{P}^i}$ with $\vec{P}^i=(p_1$, $\cdots$, $p_{i-1}$, $p'$,
$p_{i+1}$, $\cdots$, $p_m)$ and
\[
 [\vec{w}^i]_{A_{\vec{P}^i}}=[\vec{w}]_{A_{\vec{P}}}^{p_i'/p}.
\]
\end{Lemma}

\section{Proof of Theorem~\ref{thm:main}}\label{sec:s3}
In this section, we prove Theorem~\ref{thm:main}.
Without loss of generality, we assume that $f_1,f_2\ge 0$.
Denote $\Omega_l:=\{x\in\bbR^n: A^{}_{\mathscr{D},\mathcal{S}}(f_1\sigma_1,f_2\sigma_2)(x)>2^l\}$ and let $\mathcal{Q}_l$ denote the
set of maximal dyadic cubes in $\Omega_l$. By the structure of $A_{\mathscr{D},\mathcal{S}}$, any cube in $\mathcal{Q}_l$ must be some cube
$Q_{j,k}\in\mathcal{S}$.

We have
\begin{eqnarray*}
\left\|A^{}_{\mathscr{D},\mathcal{S}}(f_1\sigma_1,f_2\sigma_2)\right\|_{L^p(v_{\vec{w}})}^p&\le&
4^p\sum_{l\in\bbZ}2^{lp}v_{\vec{w}}(\Omega_{l+1}\setminus \Omega_{l+2})\\
&=&4^p\sum_{l\in\bbZ}\sum_{Q\in\mathcal{Q}_l}2^{lp}v_{\vec{w}}(Q\cap\Omega_{l+1}\setminus \Omega_{l+2})\\
&=&4^p\sum_{l\in\bbZ}\sum_{Q\in\mathcal{Q}_l}2^{lp}v_{\vec{w}}(E_l(Q)),
\end{eqnarray*}
where $E_l(Q)=Q\cap\Omega_{l+1}\setminus \Omega_{l+2}$.
By the maximal property of $Q\in\mathcal{Q}_l$, we have
\[
  \sum_{Q_{j,k}\supset Q} \prod_{i=1}^2 \frac{1}{|Q_{j,k}|}\int_{Q_{j,k}}f_i(y_i)\sigma_i dy_i >2^l
\]
and
\[
  \sum_{Q_{j,k}\supsetneq Q} \prod_{i=1}^2 \frac{1}{|Q_{j,k}|}\int_{Q_{j,k}}f_i(y_i)\sigma_i dy_i \le 2^l.
\]
Therefore, for $x\in E_l(Q)$, we have
\[
  2^{l+1}<A^{}_{\mathscr{D},\mathcal{S}}(f_1\sigma_1 , f_2\sigma_2 )(x) \le A^{}_{\mathscr{D},\mathcal{S}}(f_1\sigma_1 1_Q, f_2\sigma_2 1_Q)(x)+2^l.
\]
That is,
\[
  A^{}_{\mathscr{D},\mathcal{S}}(f_1\sigma_1 1_Q, f_2\sigma_2 1_Q)(x)>2^l,\quad\mbox{$x\in E_l(Q)$}.
\]
Thus, for sufficiently small $\beta>0$ to be determined later, we have
\begin{eqnarray*}
&&\left\|A^{}_{\mathscr{D},\mathcal{S}}(f_1\sigma_1,f_2\sigma_2)\right\|_{L^p(v_{\vec{w}})}^p\\&\le&
4^p\sum_{l\in\bbZ, Q\in\mathcal{Q}_l\atop v_{\vec{w}}(E_l(Q))\le \beta v_{\vec{w}}(Q)}2^{lp}\beta v_{\vec{w}}(Q)
+4^p\sum_{l\in\bbZ, Q\in\mathcal{Q}_l\atop v_{\vec{w}}(E_l(Q))> \beta v_{\vec{w}}(Q)}2^{lp}v_{\vec{w}}(E_l(Q))\\
&\le& \frac{4^p\beta}{1-2^{-p}}\left\|A^{}_{\mathscr{D},\mathcal{S}}(f_1\sigma_1,f_2\sigma_2)\right\|_{L^p(v_{\vec{w}})}^p
+4^p\sum_{l\in\bbZ, Q\in\mathcal{Q}_l\atop v_{\vec{w}}(E_l(Q))> \beta v_{\vec{w}}(Q)}v_{\vec{w}}(E_l(Q))^{1-p}\\
&&\quad\cdot\left(\int_{E_l(Q)}A^{}_{\mathscr{D},\mathcal{S}}(f_1\sigma_1 1_Q, f_2\sigma_2 1_Q)(x) v_{\vec{w}}dx \right)^p.
\end{eqnarray*}
Consequently, by setting $\beta =4^{-p}(1-2^{-p})/2$, we get
\begin{eqnarray}
&&\left\|A^{}_{\mathscr{D},\mathcal{S}}(f_1\sigma_1,f_2\sigma_2)\right\|_{L^p(v_{\vec{w}})}^p
 \nonumber\\
&\lesssim&
\sum_{l\in\bbZ, Q\in\mathcal{Q}_l \atop v_{\vec{w}}(E_l(Q))> \beta v_{\vec{w}}(Q)}
   \!\!\!\!\!\!\!\!v_{\vec{w}}(E_l(Q))^{1-p}
    \bigg(\int_{E_l(Q)}A^{}_{\mathscr{D},\mathcal{S}}(f_1\sigma_1 1_Q,
       f_2\sigma_2 1_Q)(x)
      v_{\vec{w}}dx \bigg)^p \nonumber\\
&&\hskip 50mm \label{eq:e7}\\
&\lesssim&\sum_{l\in\bbZ, Q\in\mathcal{Q}_l \atop v_{\vec{w}}(E_l(Q))> \beta
        v_{\vec{w}}(Q)}
        v_{\vec{w}}(E_l(Q))^{1-p}
            \nonumber
     \\
&& \qquad\qquad\qquad\times
   \bigg(\int_{E_l(Q)}A^{}_{\mathscr{D},\mathcal{S}}
      (f_1\sigma_1 1_{Q\setminus \Omega_{l+2}}, f_2\sigma_2 1_{Q\setminus \Omega_{l+2}})(x)
          v_{\vec{w}}dx \bigg)^p \nonumber\\
&& +
  \sum_{l\in\bbZ, Q\in\mathcal{Q}_l\atop v_{\vec{w}}(E_l(Q))> \beta v_{\vec{w}}(Q)}
  v_{\vec{w}}(E_l(Q))^{1-p}
     \nonumber  \\
&&\qquad\qquad\qquad
  \times \bigg(\int_{E_l(Q)}A^{}_{\mathscr{D},\mathcal{S}}(f_1\sigma_1 1_{Q\setminus
     \Omega_{l+2}}, f_2\sigma_2 1_{Q\cap \Omega_{l+2}})(x) v_{\vec{w}}dx \bigg)^p\nonumber
  \\
&&
  +\sum_{l\in\bbZ, Q\in\mathcal{Q}_l\atop v_{\vec{w}}(E_l(Q))> \beta v_{\vec{w}}(Q)}
    v_{\vec{w}}(E_l(Q))^{1-p} \nonumber
      \\
&&\qquad\qquad\qquad\times
\bigg(\int_{E_l(Q)}A^{}_{\mathscr{D},\mathcal{S}}(f_1\sigma_1 1_{Q\cap \Omega_{l+2}},
      f_2\sigma_2 1_{Q\setminus\Omega_{l+2}})(x) v_{\vec{w}}dx \bigg)^p
      \nonumber
  \\
&&+\sum_{l\in\bbZ, Q\in\mathcal{Q}_l\atop v_{\vec{w}}(E_l(Q))> \beta v_{\vec{w}}(Q)}
  v_{\vec{w}}(E_l(Q))^{1-p}
   \nonumber
       \\
&&\qquad\qquad\qquad\times
  \bigg(\int_{E_l(Q)}A^{}_{\mathscr{D},\mathcal{S}}(f_1\sigma_1 1_{Q\cap \Omega_{l+2}},
      f_2\sigma_2 1_{Q\cap \Omega_{l+2}})(x) v_{\vec{w}}dx \bigg)^p
      \nonumber\\
&:=&I_1+I_2+I_3+I_4.\nonumber
\end{eqnarray}

In the rest of this section, we use the following \textbf{conventions}.

\begin{enumerate}
\item all $l$ in these sums are of the same parity; after all, there are just two such sums;
\item keep in mind the restriction
\begin{equation}\label{eq:beta}
  v_{\vec{w}}(E_l(Q))> \beta v_{\vec{w}}(Q)
\end{equation}
appearing in all these sums
      and we omit it in the rest of this section;
\item by the monotone convergence theorem, we may also assume that all
      appearing cubes are contained in some maximal dyadic cube $\overline{Q}$. Then we can use the technique of principal cubes.
\end{enumerate}

Before further estimates, we give two lemmas. The first can be proved with similar
arguments as that in \cite[pp. 20-21]{HLMORSU} and we omit the details.
\begin{Lemma}
Let $\mathcal{G}$ be the principal cubes with respect  to $f_1$ and $\sigma_1$,
and $\tilde{\mathcal{G}}$ be the principal cubes
with respect to $f_2$ and $\sigma_2$.
Suppose that $\Gamma(Q)$ and $\tilde{\Gamma}(Q)$ are defined as that in Definition~\ref{def:d1}. Then
\begin{eqnarray}
&&\sum_{l\in\bbZ}\sum_{Q\in\mathcal{Q}_l}\sum_{R\in \mathcal{Q}_{l+2}, R\subset Q\atop \mathbb{E}_{R}^{\sigma_1}f_1> 16 \mathbb{E}_{\Gamma(Q)}^{\sigma_1}f_1}
 \sigma_1(R)(\mathbb{E}_{R}^{\sigma_1}f_1)^{p_1}\lesssim \|f_1\|_{L^{p_1}(\sigma_1)}^{p_1}. \label{eq:e9}\\
 && \sum_{l\in\bbZ}\sum_{Q\in\mathcal{Q}_l}\sum_{\tilde{R}\in\mathcal{Q}_{l+2},\tilde{R}\subset Q\atop \mathbb{E}_{\tilde{R}}^{\sigma_2}f_2> 16 \mathbb{E}_{\tilde{\Gamma}(Q)}^{\sigma_2}f_2}\sigma_2(\tilde{R})(\mathbb{E}_{\tilde{R}}^{\sigma_2}f_2)^{p_2}\lesssim \|f_2\|_{L^{p_2}(\sigma_2)}^{p_2}.\label{eq:e10}
\end{eqnarray}
\end{Lemma}

And the second one can be seen from the definition of principal cubes.

\begin{Lemma}
Let $\mathcal{G}$ be the principal cubes with respect to $f_1$ and $\sigma_1$
and  $\tilde{\mathcal{G}}$ be the principal cubes
with respect to $f_2$ and $\sigma_2$. Then
\begin{eqnarray}
&&\sum_{G\in\mathcal{G}}(\mathbb{E}_G^{\sigma_1}f_1)^{p_1}\sigma_1(G)\lesssim \|f_1\|_{L^{p_1}(\sigma_1)}^{p_1}.\label{eq:e14}\\
&&\sum_{\tilde{G}\in\tilde{\mathcal{G}}}(\mathbb{E}_{\tilde{G}}^{\sigma_2}f_2)^{p_2}\sigma_2(\tilde{G})\lesssim \|f_2\|_{L^{p_2}(\sigma_2)}^{p_2}.\label{eq:e15}
\end{eqnarray}
\end{Lemma}

Next we give a proof for Lemma~\ref{lm:l2}.

\begin{proof}[Proof of Lemma~\ref{lm:l2}]
Without loss of generality, assume that $f_1\ge 0$.  Set
\[
  A^{1}_{\mathscr{D},\mathcal{S}}(f_1\sigma_1, 1_S\sigma_2)=\sum_{Q_{j,k}\supsetneq S}\frac{\int_{Q_{j,k}}f_1(y_1)\sigma_1dy_1 \sigma_2(Q_{j,k}\cap S)}{|Q_{j,k}|^2}\chi^{}_{Q_{j,k}}
\]
and
\[
   A^{2}_{\mathscr{D},\mathcal{S}}(f_1\sigma_1, 1_S\sigma_2)=\sum_{Q_{j,k}\subset S}\frac{\int_{Q_{j,k}}f_1(y_1)\sigma_1dy_1 \sigma_2(Q_{j,k}\cap S)}{|Q_{j,k}|^2}\chi^{}_{Q_{j,k}}.
\]
It is easy to see that
\begin{eqnarray*}
1_S A^{1}_{\mathscr{D},\mathcal{S}}(f_1\sigma_1, 1_S\sigma_2)&=& \sum_{Q_{j,k}\supsetneq S}\frac{\int_{Q_{j,k}}f_1(y_1)\sigma_1dy_1 \sigma_2(Q_{j,k}\cap S)}{|Q_{j,k}|^2}\chi^{}_{S}\\
&\lesssim&\frac{\int_{S}f_1(y_1)\sigma_1dy_1 \sigma_2( S)}{|S|^2}\chi^{}_{S}\\
&\le&\frac{\|f_1\|_{L^{p_1}(\sigma_1)}\sigma_1(S)^{1/{p_1'}}\sigma_2(S)}{|S|^2}\chi^{}_{S}.
\end{eqnarray*}
Hence
\begin{eqnarray*}
\left\|1_S A^1_{\mathscr{D},\mathcal{S}}(f_1\sigma_1 , \sigma_2 1_S)\right\|_{L^p(v_{\vec{w}})}
&\lesssim&\frac{\|f_1\|_{L^{p_1}(\sigma_1)}\sigma_1(S)^{1/{p_1'}}\sigma_2(S)}{|S|^2}v_{\vec{w}}(S)^{1/{p}}\\
&\le& [\vec{w}]_{A_{\vec{P}}}^{1/p}\|f_1\|_{L^{p_1}(\sigma_1)}\sigma_2(S)^{1/{p_2}}.
\end{eqnarray*}
It remains to estimate $A^{2}_{\mathscr{D},\mathcal{S}}(|f_1|\sigma_1, 1_S\sigma_2)$. Without loss of generality, assume that
all cubes in $\mathcal{S}$ are contained in $S$.
By the previous arguments,
we only need to estimate (\ref{eq:e7}) in the special case $f_2=1_S$.
We have
\begin{eqnarray*}
&&\left\|A^{}_{\mathscr{D},\mathcal{S}}(f_1\sigma_1,1_S\sigma_2)\right\|_{L^p(v_{\vec{w}})}^p\\
&\lesssim&
\sum_{l\in\bbZ, Q\in\mathcal{Q}_l \atop v_{\vec{w}}(E_l(Q))> \beta v_{\vec{w}}(Q)}
   \!\!\!\!\!\!\!\!v_{\vec{w}}(E_l(Q))^{1-p}
    \bigg(\int_{E_l(Q)}A^{}_{\mathscr{D},\mathcal{S}}(f_1\sigma_1 1_Q,
       1_S\sigma_2 1_Q)(x)
      v_{\vec{w}}dx \bigg)^p\\
&\lesssim&\sum_{l\in\bbZ, Q\in\mathcal{Q}_l \atop v_{\vec{w}}(E_l(Q))> \beta
        v_{\vec{w}}(Q)}
        v_{\vec{w}}(E_l(Q))^{1-p}
     \\
&& \qquad\qquad\qquad\times
   \bigg(\int_{E_l(Q)}A^{}_{\mathscr{D},\mathcal{S}}
      (f_1\sigma_1 1_{Q\setminus \Omega_{l+2}}, 1_S\sigma_2 1_Q)(x)
          v_{\vec{w}}dx \bigg)^p\\
&&\quad+\sum_{l\in\bbZ, Q\in\mathcal{Q}_l \atop v_{\vec{w}}(E_l(Q))> \beta
        v_{\vec{w}}(Q)}
        v_{\vec{w}}(E_l(Q))^{1-p}
     \\
&& \qquad\qquad\qquad\times
   \bigg(\int_{E_l(Q)}A^{}_{\mathscr{D},\mathcal{S}}
      (f_1\sigma_1 1_{Q\cap \Omega_{l+2}}, 1_S\sigma_2 1_Q)(x)
          v_{\vec{w}}dx \bigg)^p\\
&:=& J_1+J_2.
\end{eqnarray*}

In the following, we also use the convention (ii) to omit $v_{\vec{w}}(E_l(Q))> \beta
        v_{\vec{w}}(Q)$.

First, we estimate $J_1$.
Let
\[
\mathcal{S}_l(Q):=  \bigcup_{ R\in \mathcal{Q}_{l+2}}\{Q_{j,k}\in \mathcal{S}: R\subsetneq Q_{j,k}\subset Q\}, \qquad \mbox{ if }  \mathcal{Q}_{l+2}\ne \emptyset   \\
\]
and
\[
\mathcal{S}_l(Q):=  \{Q_{j,k}\in \mathcal{S}:  Q_{j,k}\subset Q\},\qquad \mbox{ if }  \mathcal{Q}_{l+2} = \emptyset.
\]
By Lemma~\ref{lm:sparse}, $\mathcal{S}_l(Q)$ is sparse. For big cubes $Q_{j,k}\supsetneq Q$ and $x\in E_l(Q)$,
\[
  \sum_{Q_{j,k}\supsetneq Q}\frac{\int_{Q\setminus \Omega_{l+2}}f_1\sigma_1dx\cdot\sigma_2(Q)}{|Q_{j,k}|^2}\le
  \frac{\int_{Q\setminus \Omega_{l+2}}f_1\sigma_1dx\cdot\sigma_2(Q)}{|Q|^2}.
\]
Hence for $x\in E_l(Q)$, we have
\begin{equation}\label{eq:e20}
A^{}_{\mathscr{D},\mathcal{S}}(f_1\sigma_1 1_{Q\setminus \Omega_{l+2}}, \sigma_2 1_{Q})(x)
  \le 2A^{}_{\mathscr{D},\mathcal{S}_l(Q)}(f_1\sigma_1 1_{Q\setminus \Omega_{l+2}}, \sigma_2 1_{Q})(x),
\end{equation}
where
\[
  A^{}_{\mathscr{D},\mathcal{S}_l(Q)}(f_1, f_2)(x)=\sum_{Q_{j,k}\in\mathcal{S}_l(Q)}\bigg( \prod_{i=1}^2 \frac{1}{|Q_{j,k}|}\int_{Q_{j,k}}f_i(y_i)dy_i\bigg)\chi^{}_{Q_{j,k}}(x).
\]

Observe that $\mathcal{S}_{l}(Q)$ is an empty set if and only if $Q\in \mathcal{Q}_{l} \cap \mathcal{Q}_{l+2}$.
When the condition is satisfied,  we have
\[
  A^{}_{\mathscr{D},\mathcal{S}}(f_1\sigma_1 1_{Q\setminus \Omega_{l +2}}, \sigma_2 1_{Q})(x)
  =A^{}_{\mathscr{D},\mathcal{S}_{l}(Q)}(f_1\sigma_1 1_{Q\setminus \Omega_{l+2}}, \sigma_2 1_{Q})(x)=0,
\]
which means that (\ref{eq:e20}) is correct even if $\mathcal{S}_{l}(Q)$ is empty.
So  $\{\mathcal{S}_l(Q): l\in 2\bbZ, Q\in\mathcal{Q}_l\}$
and $\{\mathcal{S}_l(Q): l\in 2\bbZ+1, Q\in\mathcal{Q}_l\}$
 are pairwise disjoint, 
 respectively.
Denote $\mathcal{S}_l(Q) =\{Q_{l,\eta}\}_\eta$. We have
\begin{eqnarray*}
&&\int_{E_l(Q)}A^{}_{\mathscr{D},\mathcal{S}}(f_1\sigma_1 1_{Q\setminus \Omega_{l+2}}, 1_S\sigma_2 1_Q)(x) v_{\vec{w}}dx\\
&\lesssim& \int_{E_l(Q)}A^{}_{\mathscr{D},\mathcal{S}_l(Q)}(f_1\sigma_1 1_{Q\setminus \Omega_{l+2}}, 1_S\sigma_2 1_Q)(x) v_{\vec{w}}dx\\
&\le&\bigg(\int_{Q }(A^{}_{\mathscr{D},\mathcal{S}_l(Q)}( v_{\vec{w}} 1_{Q}, \sigma_2 1_Q ))^{p_1'}\sigma_1dx\bigg)^{1/{p_1'}}\cdot\bigg(\int_{Q\setminus \Omega_{l+2}}f_1^{p_1}\sigma_1\bigg)^{1/{p_1}}\\
&\lesssim& [\vec{w}]_{A_{\vec{P}}}^{1/p}
 \bigg(\sum_{Q_{l,\eta}\in\mathcal{S}_l(Q)}v_{\vec{w}}(Q_{l,\eta})\bigg)^{1/{p'}}
  \bigg(\sum_{Q_{l,\eta}\in\mathcal{S}_l(Q)}\sigma_2(Q_{l,\eta})\bigg)^{1/{p_2}} \\
&&\qquad \times
 \bigg(\int_{Q\setminus \Omega_{l+2}}f_1^{p_1}\sigma_1\bigg)^{1/{p_1}}\quad\mbox{(by Lemma~\ref{lm:l1} and Lemma~\ref{lm:w})}\\
 &\le& [\vec{w}]_{A_{\vec{P}}}^{1/p} [v_{\vec{w}}]_{A_\infty}^{1/{p'}}v_{\vec{w}}(Q)^{1/{p'}}\bigg(\sum_{Q_{l,\eta}\in\mathcal{S}_l(Q)}\sigma_2(Q_{l,\eta})\bigg)^{1/{p_2}}
 \bigg(\int_{Q\setminus \Omega_{l+2}}f_1^{p_1}\sigma_1\bigg)^{1/{p_1}},
\end{eqnarray*}
where Lemma~\ref{lm:l0} is used in the last step.
Recall that we have the convention (\ref{eq:beta}).
By H\"{o}lder's inequality, we get
\begin{eqnarray*}
&&\hskip -2em J_1 =\sum_{l\in\bbZ,  Q\in\mathcal{Q}_l}
\bigg(\int_{E_l(Q)}A^{}_{\mathscr{D},\mathcal{S}}(f_1\sigma_1 1_{Q\setminus \Omega_{l+2}}, 1_S\sigma_2 1_Q)(x) v_{\vec{w}}dx \bigg)^p
  v_{\vec{w}}(E_l(Q))^{1-p}\\
&\lesssim&[\vec{w}]_{A_{\vec{P}}} [v_{\vec{w}}]_{A_\infty}^{p/{p'}}
  \bigg(
  \sum_{l\in\bbZ \atop Q\in\mathcal{Q}_l}
  \sum_{Q_{l,\eta}\in\mathcal{S}_l(Q)}
  \sigma_2(Q_{l,\eta})\bigg)^{p/p_2}
 \bigg(\sum_{l\in\bbZ \atop Q\in\mathcal{Q}_l}\int_{Q\setminus \Omega_{l+2}}f_1^{p_1}\sigma_1\bigg)^{p/p_1}\\
&\le&[\vec{w}]_{A_{\vec{P}}} [v_{\vec{w}}]_{A_{\infty}}^{p/{p'}}[\sigma_2]_{A_\infty}^{p/{p_2}}\sigma_2(S)^{p/{p_2}} \|f_1\|_{L^{p_1}(\sigma_1)}^p.
\end{eqnarray*}

Next we estimate $J_2$.
Since $E_l(Q)\subset \Omega_{l+2}^c$  for $R\in\mathcal{Q}_{l+2}$ with $R\subset Q$,
$A^{}_{\mathscr{D},\mathcal{S}}(v_{\vec{w}}1_{E_l(Q)}, \sigma_2 1_{Q})(x)$ is a constant for $x\in R$.
We have
\begin{eqnarray}
&&\int_{E_l(Q)}A^{}_{\mathscr{D},\mathcal{S}}(f_1\sigma_1 1_{Q\cap \Omega_{l+2}}, 1_S\sigma_2 1_{Q})(x) v_{\vec{w}}dx
  \nonumber \\
&=&\sum_{R\in\mathcal{Q}_{l+2}\atop R\subset Q}
\int_{R}A^{}_{\mathscr{D},\mathcal{S}}(v_{\vec{w}}1_{E_l(Q)}, \sigma_2 1_{Q})(x) \sigma_1dx\cdot \mathbb{E}_R^{\sigma_1}f_1\nonumber \\
&\le&16\sum_{R\in\mathcal{Q}_{l+2}, R\subset Q\atop \mathbb{E}_R^{\sigma_1}f_1\le 16 \mathbb{E}_{\Gamma(Q)}^{\sigma_1}f_1}
\int_{R}A^{}_{\mathscr{D},\mathcal{S}}(v_{\vec{w}}1_{E_l(Q)}, \sigma_2 1_{Q})(x) \sigma_1dx\cdot \mathbb{E}_{\Gamma(Q)}^{\sigma_1}f_1\nonumber \\
&&\quad+\sum_{R\in\mathcal{Q}_{l+2}, R\subset Q\atop \mathbb{E}_R^{\sigma_1}f_1> 16 \mathbb{E}_{\Gamma(Q)}^{\sigma_1}f_1}
\int_{R}A^{}_{\mathscr{D},\mathcal{S}}(v_{\vec{w}}1_{E_l(Q)}, \sigma_2 1_{Q})(x) \sigma_1dx\cdot \mathbb{E}_{R}^{\sigma_1}f_1\nonumber \\
&\le& 16
\int_{E_l(Q)}A^{}_{\mathscr{D},\mathcal{S}}(\sigma_1 1_{Q\cap \Omega_{l+2}}, \sigma_2 1_{Q})(x) v_{\vec{w}}dx\cdot \mathbb{E}_{\Gamma(Q)}^{\sigma_1}f_1\nonumber \\
&&\quad+\sum_{R\in\mathcal{Q}_{l+2}, R\subset Q\atop \mathbb{E}_R^{\sigma_1}f_1> 16 \mathbb{E}_{\Gamma(Q)}^{\sigma_1}f_1}
\int_{R}A^{}_{\mathscr{D},\mathcal{S}}(v_{\vec{w}}1_{E_l(Q)}, \sigma_2 1_{Q})(x) \sigma_1dx\cdot \mathbb{E}_{R}^{\sigma_1}f_1.
  \label{eq:J2:1}
\end{eqnarray}
Similarly to (\ref{eq:e20}), for $x\in E_l(Q)$, we have
\[
  A^{}_{\mathscr{D},\mathcal{S}}(f_1\sigma_1 1_{Q\cap \Omega_{l+2}}, \sigma_2 1_{Q})(x)
  \le 2A^{}_{\mathscr{D},\mathcal{S}_l(Q)}(f_1\sigma_1 1_{Q\cap \Omega_{l+2}}, \sigma_2 1_{Q})(x).
\]
Consequently, by setting $ \mathcal S(G)=\bigcup_{l\in\bbZ, Q\in\mathcal{Q}_l\atop \Gamma(Q)=G}\mathcal{S}_l(Q)$,  we have
\begin{eqnarray}
&&\sum_{l\in\bbZ, Q\in\mathcal{Q}_l}\bigg(\int_{E_l(Q)}A^{}_{\mathscr{D},\mathcal{S}}(\sigma_1 1_{Q\cap \Omega_{l+2}}, \sigma_2 1_{Q})(x) v_{\vec{w}}dx\bigg)^p\nonumber \\
&&\quad\cdot (\mathbb{E}_{\Gamma(Q)}^{\sigma_1}f_1)^p\cdot v_{\vec{w}}(E_l(Q))^{1-p}\nonumber \\
&\lesssim& \sum_{l\in\bbZ, Q\in\mathcal{Q}_l}\int_{E_l(Q)}(A^{}_{\mathscr{D},\mathcal{S}_l(Q)}(\sigma_1 1_{Q\cap \Omega_{l+2}}, \sigma_2 1_{Q}))^p v_{\vec{w}}dx \cdot (\mathbb{E}_{\Gamma(Q)}^{\sigma_1}f_1)^p\nonumber \\
&=&\sum_{G\in\mathcal{G}}\sum_{l\in\bbZ, Q\in\mathcal{Q}_l\atop \Gamma(Q)=G}
\int_{E_l(Q)}(A^{}_{\mathscr{D},\mathcal{S}_l(Q)}(\sigma_1 1_{Q\cap \Omega_{l+2}}, \sigma_2 1_{Q}))^p v_{\vec{w}}dx \cdot (\mathbb{E}_{G}^{\sigma_1}f_1)^p\nonumber \\
&\le&\sum_{G\in\mathcal{G}}
\int_{G}(A^{}_{\mathscr{D},\mathcal{S}(G)}(\sigma_1 1_{G}, \sigma_2 1_{G}))^p v_{\vec{w}}dx \cdot (\mathbb{E}_{G}^{\sigma_1}f_1)^p \nonumber \\
&&\hskip 30mm\mbox{(by the disjointness of $E_l(Q)$)}\nonumber \\
&\lesssim& \sum_{G\in\mathcal{G}}[\vec{w}]_{A_{\vec{P}}}\bigg(\sum_{l\in\bbZ, Q\in\mathcal{Q}_l\atop \Gamma(Q)=G}\sum_{Q_{l,\eta}\in\mathcal{S}_l(Q)}\sigma_1(Q_{l,\eta})\bigg)^{p/{p_1}}\nonumber \\
&&\quad\times\bigg(\sum_{l\in\bbZ, Q\in\mathcal{Q}_l\atop \Gamma(Q)=G}\sum_{Q_{l,\eta}\in\mathcal{S}_l(Q)}\sigma_2(Q_{l,\eta})\bigg)^{p/{p_2}}
\cdot (\mathbb{E}_{G}^{\sigma_1}f_1)^p\quad\mbox{(by Lemma~\ref{lm:l1})}\nonumber \\
&\le& [\vec{w}]_{A_{\vec{P}}}\bigg(\sum_{G\in\mathcal{G}}\sum_{l\in\bbZ, Q\in\mathcal{Q}_l\atop \Gamma(Q)=G}\sum_{Q_{l,\eta}\in\mathcal{S}_l(Q)}\sigma_1(Q_{l,\eta})\cdot (\mathbb{E}_{G}^{\sigma_1}f_1)^{p_1}\bigg)^{p/{p_1}}\nonumber \\
&&\quad\times\bigg(\sum_{G\in\mathcal{G}}\sum_{l\in\bbZ, Q\in\mathcal{Q}_l\atop \Gamma(Q)=G}\sum_{Q_{l,\eta}\in\mathcal{S}_l(Q)}\sigma_2(Q_{l,\eta})\bigg)^{p/{p_2}}
\nonumber \\
&\lesssim&[\vec{w}]_{A_{\vec{P}}}
\bigg(\sum_{G\in\mathcal{G}}\int_G M(\sigma_1 1_G)
\cdot (\mathbb{E}_{G}^{\sigma_1}f_1)^{p_1}\bigg)^{p/{p_1}}\cdot\bigg(\int_S M(\sigma_2 1_S)\bigg)^{p/{p_2}}\nonumber \\
&&\hskip 60mm \mbox{(by Lemma~\ref{lm:l0})}\nonumber \\
&\le& [\vec{w}]_{A_{\vec{P}}}[\sigma_1]_{A_\infty}^{p/{p_1}}[\sigma_2]_{A_\infty}^{p/{p_2}}\sigma_2(S)^{p/{p_2}}
\cdot\bigg(\sum_{G\in\mathcal{G}}\sigma_1(G)
\cdot (\mathbb{E}_{G}^{\sigma_1}f_1)^{p_1}\bigg)^{p/{p_1}}\nonumber \\
&\lesssim& [\vec{w}]_{A_{\vec{P}}}[\sigma_1]_{A_\infty}^{p/{p_1}}[\sigma_2]_{A_\infty}^{p/{p_2}}\sigma_2(S)^{p/{p_2}}
\cdot \|f_1\|_{L^{p_1}(\sigma_1)}^p,\label{eq:J2:2}
\end{eqnarray}
where (\ref{eq:e14}) is used in the last step.

For the $\mathbb{E}_R^{\sigma_1}f_1> 16 \mathbb{E}_{\Gamma(Q)}^{\sigma_1}f_1$ part, by H\"older's inequality, we have
\begin{eqnarray*}
&&\sum_{R\in\mathcal{Q}_{l+2}, R\subset Q\atop \mathbb{E}_R^{\sigma_1}f_1> 16 \mathbb{E}_{\Gamma(Q)}^{\sigma_1}f_1}
\int_{R}A^{}_{\mathscr{D},\mathcal{S}}(v_{\vec{w}}1_{E_l(Q)}, \sigma_2 1_{Q})(x) \sigma_1dx\cdot \mathbb{E}_{R}^{\sigma_1}f_1\\
&\le&\bigg(\!
  \sum_{R\in\mathcal{Q}_{l+2}, R\subset Q\atop \mathbb{E}_R^{\sigma_1}f_1> 16
    \mathbb{E}_{\Gamma(Q)}^{\sigma_1}f_1} \!\!\!
     \sigma_1(R)^{-p_1'/{p_1}}
      \bigg(\int_{R}A^{}_{\mathscr{D},\mathcal{S}}(v_{\vec{w}}1_{E_l(Q)}, \sigma_2 1_{Q})(x)
        \sigma_1dx\bigg)^{p_1'}\bigg)^{1/{p_1'}}\\
&&\quad\times \bigg(\sum_{R\in\mathcal{Q}_{l+2}, R\subset Q\atop \mathbb{E}_R^{\sigma_1}f_1> 16 \mathbb{E}_{\Gamma(Q)}^{\sigma_1}f_1}\sigma_1(R)(\mathbb{E}_{R}^{\sigma_1}f_1)^{p_1}\bigg)^{1/{p_1}}\\
&\le&\bigg(\int_{Q\cap \Omega_{l+2}}(A^{}_{\mathscr{D},\mathcal{S}}(v_{\vec{w}} 1_{E_l(Q)} ,\sigma_2 1_{Q}))^{p_1'}\sigma_1 dx \bigg)^{1/{p_1'}}\\
&&\quad\times \bigg(\sum_{R\in\mathcal{Q}_{l+2}, R\subset Q\atop \mathbb{E}_R^{\sigma_1}f_1> 16 \mathbb{E}_{\Gamma(Q)}^{\sigma_1}f_1}\sigma_1(R)(\mathbb{E}_{R}^{\sigma_1}f_1)^{p_1}\bigg)^{1/{p_1}}\\
&\lesssim& \bigg(\int_{Q\cap \Omega_{l+2}}(A^{}_{\mathscr{D},\mathcal{S}_l(Q)}(v_{\vec{w}} 1_{E_l(Q)} ,\sigma_2 1_{Q}))^{p_1'}\sigma_1 dx \bigg)^{1/{p_1'}}\\
&&\quad\times \bigg(\sum_{R\in\mathcal{Q}_{l+2}, R\subset Q\atop \mathbb{E}_R^{\sigma_1}f_1> 16 \mathbb{E}_{\Gamma(Q)}^{\sigma_1}f_1}\sigma_1(R)(\mathbb{E}_{R}^{\sigma_1}f_1)^{p_1}\bigg)^{1/{p_1}}\\
&\lesssim& [\vec{w}]_{A_{\vec{P}}}^{1/p}
 \bigg(\sum_{Q_{l,\eta}\in\mathcal{S}_l(Q)}v_{\vec{w}}(Q_{l,\eta})\bigg)^{1/{p'}}\cdot\bigg(\sum_{Q_{l,\eta}\in\mathcal{S}_l(Q)}\sigma_2(Q_{l,\eta})\bigg)^{1/{p_2}}\\
&&\quad\times \bigg(\sum_{R\in\mathcal{Q}_{l+2}, R\subset Q\atop \mathbb{E}_R^{\sigma_1}f_1> 16 \mathbb{E}_{\Gamma(Q)}^{\sigma_1}f_1}\sigma_1(R)(\mathbb{E}_{R}^{\sigma_1}f_1)^{p_1}\bigg)^{1/{p_1}}\\
&&\hskip 40mm \mbox{(by Lemma~\ref{lm:l1} and Lemma~\ref{lm:w})}\\
&\lesssim& [\vec{w}]_{A_{\vec{P}}}^{1/p}[v_{\vec{w}}]_{A_\infty}^{1/{p'}}v_{\vec{w}}(Q)^{1/{p'}}\cdot\bigg(\sum_{Q_{l,\eta}\in\mathcal{S}_l(Q)}\sigma_2(Q_{l,\eta})\bigg)^{1/{p_2}}\\
&&\quad\times \bigg(\sum_{R\in\mathcal{Q}_{l+2}, R\subset Q\atop \mathbb{E}_R^{\sigma_1}f_1> 16 \mathbb{E}_{\Gamma(Q)}^{\sigma_1}f_1}\sigma_1(R)(\mathbb{E}_{R}^{\sigma_1}f_1)^{p_1}\bigg)^{1/{p_1}},
\end{eqnarray*}
where Lemma~\ref{lm:l0} is used in the last step. It follows from (\ref{eq:beta}) and H\"older's inequality that
\begin{eqnarray}
&&\hskip -2.5em\sum_{l\in\bbZ, Q\in\mathcal{Q}_l}\!\!\!\bigg( \!\!\sum_{R\in\mathcal{Q}_{l+2}, R\subset Q\atop \mathbb{E}_R^{\sigma_1}f_1> 16 \mathbb{E}_{\Gamma(Q)}^{\sigma_1}f_1}
\!\!\!\!\int_{R}A^{}_{\mathscr{D},\mathcal{S}}(v_{\vec{w}}1_{E_l(Q)}, \sigma_2 1_{Q})(x) \sigma_1dx \!\cdot\!
  \mathbb{E}_{R}^{\sigma_1}f_1\bigg)^p\!\! v_{\vec{w}}(E_l(Q))^{1-p}\nonumber \\
&\lesssim& [\vec{w}]_{A_{\vec{P}}}[v_{\vec{w}}]_{A_\infty}^{p/{p'}}\cdot\bigg(\sum_{l\in\bbZ, Q\in\mathcal{Q}_l}\sum_{Q_{l,\eta}\in\mathcal{S}_l(Q)}\sigma_2(Q_{l,\eta})\bigg)^{p/{p_2}}\nonumber \\
&&\quad\times \bigg(\sum_{l\in\bbZ, Q\in\mathcal{Q}_l}\sum_{R\in\mathcal{Q}_{l+2}, R\subset Q\atop \mathbb{E}_R^{\sigma_1}f_1> 16 \mathbb{E}_{\Gamma(Q)}^{\sigma_1}f_1}\sigma_1(R)(\mathbb{E}_{R}^{\sigma_1}f_1)^{p_1}\bigg)^{p/{p_1}}\nonumber \\
&\lesssim& [\vec{w}]_{A_{\vec{P}}}[v_{\vec{w}}]_{A_\infty}^{p/{p'}}[\sigma_2]_{A_\infty}^{p/{p_2}}\sigma_2(S)^{p/{p_2}}\|f_1\|_{L^{p_1}(\sigma_1)}^p,
  \label{eq:J2:3}
\end{eqnarray}
where Lemma~\ref{lm:l0} and (\ref{eq:e9}) are used in the last step.

Putting (\ref{eq:J2:1}), (\ref{eq:J2:2}) and (\ref{eq:J2:3}) together, we get
\[
  J_2 \le [\vec{w}]_{A_{\vec{P}}} ( [v_{\vec{w}}]_{A_\infty}^{p/{p'}} + [\sigma_1]_{A_\infty}^{p/{p_1}} )
     [\sigma_2]_{A_\infty}^{p/{p_2}}\sigma_2(S)^{p/{p_2}}\|f_1\|_{L^{p_1}(\sigma_1)}^p.
\]
This completes the proof.
\end{proof}

Now we continue to prove Theorem~\ref{thm:main}.

\subsection{Estimate of $I_1$}
In this subsection, we consider $I_1$. we have
\begin{eqnarray*}
&&\int_{E_l(Q)}A^{}_{\mathscr{D},\mathcal{S}}(f_1\sigma_1 1_{Q\setminus \Omega_{l+2}}, f_2\sigma_2 1_{Q\setminus \Omega_{l+2}})(x) v_{\vec{w}}dx\\
&=&\int_{Q\setminus \Omega_{l+2}}A^{}_{\mathscr{D},\mathcal{S}}(f_1\sigma_1 1_{Q\setminus \Omega_{l+2}}, v_{\vec{w}} 1_{E_l(Q)})(x)  f_2\sigma_2dx\\
&\le&
  \bigg(\int_{Q\setminus \Omega_{l+2}} \!\! (A^{}_{\mathscr{D},\mathcal{S}}
     (f_1\sigma_1 1_{Q\setminus \Omega_{l+2}}, v_{\vec{w}}
     1_{Q}))^{p_2'}
  \sigma_2dx\bigg)^{1/{p_2'}}\!\cdot\bigg(\int_{Q\setminus      \Omega_{l+2}}
     \! f_2^{p_2}\sigma_2\bigg)^{1/{p_2}}
     \\
&\lesssim&[\vec{w}]_{A_{\vec{P}}}^{1/p}
  [v_{\vec{w}}]_{A_\infty}^{1/{p'}}([\sigma_1]_{A_\infty}^{1/{p_1}}+[\sigma_2]_{A_\infty}^{1/{p_2}})\cdot v_{\vec{w}}(Q)^{1/{p'}}\\
  &&\quad\times \bigg(\int_{Q\setminus \Omega_{l+2}}f_1^{p_1}\sigma_1\bigg)^{1/{p_1}}\cdot\bigg(\int_{Q\setminus \Omega_{l+2}}f_2^{p_2}\sigma_2\bigg)^{1/{p_2}},
\end{eqnarray*}
where Lemma~\ref{lm:l2} and Lemma~\ref{lm:w} are used. By H\"{o}lder's inequality, we have
\begin{eqnarray*}
I_1&\lesssim&[\vec{w}]_{A_{\vec{P}}}
  [v_{\vec{w}}]_{A_\infty}^{p/{p'}}([\sigma_1]_{A_\infty}^{1/{p_1}}+[\sigma_2]_{A_\infty}^{1/{p_2}})^p\bigg(\sum_{l\in\bbZ, Q\in\mathcal{Q}_l}\int_{Q\setminus \Omega_{l+2}}f_1^{p_1}\sigma_1\bigg)^{p/{p_1}}\\
  &&\quad\times\bigg(\sum_{l\in\bbZ, Q\in\mathcal{Q}_l}\int_{Q\setminus \Omega_{l+2}}f_2^{p_2}\sigma_1\bigg)^{p/{p_2}}\\
  &\le&[\vec{w}]_{A_{\vec{P}}}
  [v_{\vec{w}}]_{A_\infty}^{p/{p'}}([\sigma_1]_{A_\infty}^{1/{p_1}}+[\sigma_2]_{A_\infty}^{1/{p_2}})^p\|f_1\|_{L^{p_1}(\sigma_1)}^p
  \cdot\|f_2\|_{L^{p_2}(\sigma_2)}^p.
\end{eqnarray*}

\subsection{Estimates of $I_2$ and $I_3$}

In this subsection, we estimate $I_2$ and $I_3$.
Since they are similar, we only estimate $I_2$ and the other one can be estimated similarly
by the  symmetry.
We have
\begin{eqnarray*}
&&\int_{E_l(Q)}A^{}_{\mathscr{D},\mathcal{S}}(f_1\sigma_1 1_{Q\setminus \Omega_{l+2}}, f_2\sigma_2 1_{Q\cap \Omega_{l+2}})(x) v_{\vec{w}}dx\\
&=&\int_{Q\cap \Omega_{l+2}}A^{}_{\mathscr{D},\mathcal{S}}(f_1\sigma_1 1_{Q\setminus \Omega_{l+2}}, v_{\vec{w}} 1_{E_l(Q)} )(x) f_2\sigma_2 dx\\
&=&\sum_{\tilde{R}\in \mathcal{Q}_{l+2}\atop \tilde{R}\subset Q}\int_{\tilde{R}}A^{}_{\mathscr{D},\mathcal{S}}(f_1\sigma_1 1_{Q\setminus \Omega_{l+2}}, v_{\vec{w}} 1_{E_l(Q)} )(x) f_2\sigma_2 dx.
\end{eqnarray*}
Since $E_l(Q)\subset \Omega_{l+2}^c$ and $\tilde{R}\in \mathcal{Q}_{l+2}$,
$A^{}_{\mathscr{D},\mathcal{S}}(f_1\sigma_1 1_{Q\setminus \Omega_{l+2}}, v_{\vec{w}} 1_{E_l(Q)} )(x)$
 is a constant for $x\in \tilde{R}$. Therefore,
\begin{eqnarray}
&&\int_{E_l(Q)}A^{}_{\mathscr{D},\mathcal{S}}(f_1\sigma_1 1_{Q\setminus \Omega_{l+2}}, f_2\sigma_2 1_{Q\cap \Omega_{l+2}})(x) v_{\vec{w}}dx\nonumber\\
&=&\sum_{\tilde{R}\in \mathcal{Q}_{l+2}\atop \tilde{R}\subset Q}\int_{\tilde{R}}A^{}_{\mathscr{D},\mathcal{S}}(f_1\sigma_1 1_{Q\setminus \Omega_{l+2}}, v_{\vec{w}} 1_{E_l(Q)} )(x)\sigma_2 dx
\cdot \mathbb{E}_{\tilde{R}}^{\sigma_2}f_2\nonumber\\
&\le&16\sum_{\tilde{R}\in \mathcal{Q}_{l+2}, \tilde{R}\subset Q\atop \mathbb{E}_{\tilde{R}}^{\sigma_2}f_2\le 16 \mathbb{E}_{\tilde{\Gamma}(Q)}^{\sigma_2}f_2}\int_{R}A^{}_{\mathscr{D},\mathcal{S}}(f_1\sigma_1 1_{Q\setminus \Omega_{l+2}}, v_{\vec{w}} 1_{E_l(Q)} )(x)\sigma_2 dx
\cdot \mathbb{E}_{\tilde{\Gamma}(Q)}^{\sigma_2}f_2\nonumber\\
&&\quad+\sum_{\tilde{R}\in \mathcal{Q}_{l+2}, \tilde{R}\subset Q\atop \mathbb{E}_{\tilde{R}}^{\sigma_2}f_2> 16 \mathbb{E}_{\tilde{\Gamma}(Q)}^{\sigma_2}f_2}\int_{R}A^{}_{\mathscr{D},\mathcal{S}}(f_1\sigma_1 1_{Q\setminus \Omega_{l+2}}, v_{\vec{w}} 1_{E_l(Q)} )(x)\sigma_2 dx
\cdot \mathbb{E}_{\tilde{R}}^{\sigma_2}f_2,\nonumber\\
&&\hskip 40mm\label{eq:e21}
\end{eqnarray}
where $\tilde{\Gamma}(Q)$ are the principal cubes with respect to $f_2$ and $\sigma_2$.

\subsubsection{The part with $\mathbb{E}_{\tilde{R}}^{\sigma_2}f_2\le 16 \mathbb{E}_{\tilde{\Gamma}(Q)}^{\sigma_2}f_2$}
For this part, we have
\begin{eqnarray*}
&&\sum_{\tilde{R}\in \mathcal{Q}_{l+2}, \tilde{R}\subset Q\atop \mathbb{E}_{\tilde{R}}^{\sigma_2}f_2\le 16 \mathbb{E}_{\tilde{\Gamma}(Q)}^{\sigma_2}f_2}\int_{\tilde{R}}A^{}_{\mathscr{D},\mathcal{S}}(f_1\sigma_1 1_{Q\setminus \Omega_{l+2}}, v_{\vec{w}} 1_{E_l(Q)} )(x)\sigma_2 dx
\cdot \mathbb{E}_{\tilde{\Gamma}(Q)}^{\sigma_2}f_2\\
&&\quad \le \int_{Q}A^{}_{\mathscr{D},\mathcal{S}}(f_1\sigma_1 1_{Q\setminus \Omega_{l+2}}, v_{\vec{w}} 1_{E_l(Q)} )(x)\sigma_2 dx
\cdot \mathbb{E}_{\tilde{\Gamma}(Q)}^{\sigma_2}f_2.
\end{eqnarray*}
For $\tilde{G}\in\tilde{\mathcal{G}}$, where $\tilde{\mathcal{G}}$ is the set consisting
of principal cubes
with respect to $f_2$ and $\sigma_2$, set
\[
  g_1(x)=\sum_{l\in\bbZ, Q\in\mathcal{Q}_l\atop \tilde{\Gamma}(Q)=\tilde{G}}f_1(x) 1_{Q\setminus \Omega_{l+2}}(x).
\]
Then by the disjointness of $E_l(Q)$ in $l\in\bbZ$ and $Q\in\mathcal{Q}_l$,
\begin{eqnarray}
&& \sum_{l\in\bbZ, Q\in\mathcal{Q}_l}\!\!\!\bigg(\!\!\sum_{\tilde{R}\in \mathcal{Q}_{l+2},\tilde{ R}\subset Q\atop \mathbb{E}_{\tilde{R}}^{\sigma_2}f_2\le 16 \mathbb{E}_{\tilde{\Gamma}(Q)}^{\sigma_2}f_2}
\int_{\tilde{R}}A^{}_{\mathscr{D},\mathcal{S}}(f_1\sigma_1 1_{Q\setminus \Omega_{l+2}}, v_{\vec{w}} 1_{E_l(Q)} )(x)\sigma_2 dx\nonumber\\
&&\qquad\qquad
\cdot \mathbb{E}_{\tilde{\Gamma}(Q)}^{\sigma_2}f_2\bigg)^p\cdot v_{\vec{w}}(E_l(Q))^{1-p}\nonumber\\
&\le&\sum_{l\in\bbZ, Q\in\mathcal{Q}_l}v_{\vec{w}}(E_l(Q))^{1-p} \left(\int_{Q}A^{}_{\mathscr{D},\mathcal{S}}(f_1\sigma_1 1_{Q\setminus \Omega_{l+2}}, v_{\vec{w}} 1_{E_l(Q)} )(x)\sigma_2 dx\right)^p  \nonumber\\
&&\qquad \qquad \times \left(\mathbb{E}_{\tilde{\Gamma}(Q)}^{\sigma_2}f_2\right)^p\nonumber\\
&\le&\sum_{\tilde{G}\in\tilde{\mathcal{G}}}\sum_{l\in\bbZ, Q\in\mathcal{Q}_l\atop \tilde{\Gamma}(Q)=\tilde{G}} \int_{E_l(Q)}(A^{}_{\mathscr{D},\mathcal{S}}(f_1\sigma_1 1_{Q\setminus \Omega_{l+2}} , \sigma_2 1_{Q} )(x))^pv_{\vec{w}} dx \cdot(\mathbb{E}_{\tilde{\Gamma}(Q)}^{\sigma_2}f_2)^p\nonumber\\
&\le& \sum_{\tilde{G}\in\tilde{\mathcal{G}}} \int_{\tilde{G}}(A^{}_{\mathscr{D},\mathcal{S}}(g_1\sigma_1 , \sigma_2 1_{\tilde{G}} )(x))^pv_{\vec{w}} dx \cdot(\mathbb{E}_{\tilde{G}}^{\sigma_2}f_2)^p\nonumber\\
&\lesssim& [\vec{w}]_{A_{\vec{P}}}[\sigma_2]_{A_\infty}^{p/{p_2}}
  ([\sigma_1]_{A_\infty}^{1/{p_1}}+[v_{\vec{w}}]_{A_\infty}^{1/{p'}})^p
    \sum_{\tilde{G}\in\tilde{\mathcal{G}}}
     \bigg(\!\int \! g_1^{p_1}\sigma_1\bigg)^{p/{p_1}} \!\!
\sigma_2(\tilde{G})^{p/{p_2}}(\mathbb{E}_{\tilde{G}}^{\sigma_2}f_2)^p\nonumber\\
&&\hskip 60mm \mbox{(be Lemma~\ref{lm:l2})}\nonumber\\
&\le&[\vec{w}]_{A_{\vec{P}}}[\sigma_2]_{A_\infty}^{p/{p_2}}([\sigma_1]_{A_\infty}^{1/{p_1}}+[v_{\vec{w}}]_{A_\infty}^{1/{p'}})^p
\bigg(\sum_{\tilde{G}\in\tilde{\mathcal{G}}}\int g_1^{p_1}\sigma_1\bigg)^{p/{p_1}} \nonumber\\
&&\qquad \times
\bigg(\sum_{\tilde{G}\in\tilde{\mathcal{G}}}\sigma_2(\tilde{G})(\mathbb{E}_{\tilde{G}}^{\sigma_2}f_2)^{p_2}\bigg)^{p/{p_2}}\nonumber\\
&\le&[\vec{w}]_{A_{\vec{P}}}[\sigma_2]_{A_\infty}^{p/{p_2}}([\sigma_1]_{A_\infty}^{1/{p_1}}+[v_{\vec{w}}]_{A_\infty}^{1/{p'}})^p\|f_1\|_{L^{p_1}(\sigma_1)}^p
  \cdot\|f_2\|_{L^{p_2}(\sigma_2)}^p,\label{eq:e22}
\end{eqnarray}
where (\ref{eq:e15}) is used in the last step.

\subsubsection{The part with $\mathbb{E}_{\tilde{R}}^{\sigma_2}f_2> 16 \mathbb{E}_{\tilde{\Gamma}(Q)}^{\sigma_2}f_2$}\label{sec:s322}
By H\"older's inequality, we have
\begin{eqnarray*}
&&\sum_{\tilde{R}\in \mathcal{Q}_{l+2},\tilde{ R}\subset Q\atop \mathbb{E}_{\tilde{R}}^{\sigma_2}f_2> 16 \mathbb{E}_{\tilde{\Gamma}(Q)}^{\sigma_2}f_2}\int_{\tilde{R}}A^{}_{\mathscr{D},\mathcal{S}}(f_1\sigma_1 1_{Q\setminus \Omega_{l+2}}, v_{\vec{w}} 1_{E_l(Q)} )(x)\sigma_2 dx
\cdot \mathbb{E}_{\tilde{R}}^{\sigma_2}f_2\\
&\le& \bigg(\sum_{\tilde{R}\in \mathcal{Q}_{l+2}\atop \tilde{R}\subset Q} \!\!\!\! \sigma_2(\tilde{R})^{-p_2'/{p_2}}\bigg(\int_{\tilde{R}}A^{}_{\mathscr{D},\mathcal{S}}(f_1\sigma_1 1_{Q\setminus \Omega_{l+2}}, v_{\vec{w}} 1_{E_l(Q)} )(x)\sigma_2 dx\bigg)^{p_2'}\bigg)^{1/{p_2'}}\\
&&\quad\cdot \bigg(\sum_{\tilde{R}\in\mathcal{Q}_{l+2},\tilde{R}\subset Q\atop \mathbb{E}_{\tilde{R}}^{\sigma_2}f_2> 16 \mathbb{E}_{\tilde{\Gamma}(Q)}^{\sigma_2}f_2}\sigma_2(\tilde{R})(\mathbb{E}_{\tilde{R}}^{\sigma_2}f_2)^{p_2}\bigg)^{1/{p_2}}\\
&\le&\bigg(\int_{Q\cap \Omega_{l+2}}(A^{}_{\mathscr{D},\mathcal{S}}(f_1\sigma_1 1_{Q\setminus \Omega_{l+2}}, v_{\vec{w}} 1_{E_l(Q)} ))^{p_2'}\sigma_2 dx \bigg)^{1/{p_2'}}\\
&&\quad\cdot \bigg(\sum_{\tilde{R}\in\mathcal{Q}_{l+2},\tilde{R}\subset Q\atop \mathbb{E}_{\tilde{R}}^{\sigma_2}f_2> 16 \mathbb{E}_{\tilde{\Gamma}(Q)}^{\sigma_2}f_2}\sigma_2(\tilde{R})(\mathbb{E}_{\tilde{R}}^{\sigma_2}f_2)^{p_2}\bigg)^{1/{p_2}}\\
&\lesssim&[\vec{w}]_{A_{\vec{P}}}^{1/p}
  [v_{\vec{w}}]_{A_\infty}^{1/{p'}}([\sigma_1]_{A_\infty}^{1/{p_1}}+[\sigma_2]_{A_\infty}^{1/{p_2}}) \bigg(\int_{Q\setminus \Omega_{l+2}}f_1^{p_1}\sigma_1\bigg)^{1/{p_1}}
v_{\vec{w}}(Q)^{1/{p'}}\\
&&\quad\cdot \bigg(\sum_{\tilde{R}\in\mathcal{Q}_{l+2},\tilde{R}\subset Q\atop \mathbb{E}_{\tilde{R}}^{\sigma_2}f_2> 16 \mathbb{E}_{\tilde{\Gamma}(Q)}^{\sigma_2}f_2}\sigma_2(\tilde{R})(\mathbb{E}_{\tilde{R}}^{\sigma_2}f_2)^{p_2}\bigg)^{1/{p_2}},
\end{eqnarray*}
where, again, Lemma~\ref{lm:l2} and Lemma~\ref{lm:w} are used in the last step.
Therefore, by H\"{o}lder's inequality, we get
\begin{eqnarray}
&&\sum_{l\in\bbZ, Q\in\mathcal{Q}_l}\bigg(\sum_{\tilde{R}\in \mathcal{Q}_{l+2},\tilde{ R}\subset Q\atop \mathbb{E}_{\tilde{R}}^{\sigma_2}f_2> 16 \mathbb{E}_{\tilde{\Gamma}(Q)}^{\sigma_2}f_2}\int_{\tilde{R}}A^{}_{\mathscr{D},\mathcal{S}}(f_1\sigma_1 1_{Q\setminus \Omega_{l+2}}, v_{\vec{w}} 1_{E_l(Q)} )(x)\sigma_2 dx\nonumber\\&&\quad
\cdot \mathbb{E}_{\tilde{R}}^{\sigma_2}f_2\bigg)^p\cdot v_{\vec{w}}(E_l(Q))^{1-p}\nonumber\\
&\lesssim& [\vec{w}]_{A_{\vec{P}}}
  [v_{\vec{w}}]_{A_\infty}^{p/{p'}}([\sigma_1]_{A_\infty}^{1/{p_1}}+[\sigma_2]_{A_\infty}^{1/{p_2}})^p \bigg(\sum_{l\in\bbZ, Q\in\mathcal{Q}_l}\int_{Q\setminus \Omega_{l+2}}f_1^{p_1}\sigma_1\bigg)^{p/{p_1}}\nonumber\\
&&\quad\cdot \bigg(\sum_{l\in\bbZ, Q\in\mathcal{Q}_l}\sum_{\tilde{R}\in\mathcal{Q}_{l+2},\tilde{R}\subset Q\atop \mathbb{E}_{\tilde{R}}^{\sigma_2}f_2> 16 \mathbb{E}_{\tilde{\Gamma}(Q)}^{\sigma_2}f_2}\sigma_2(\tilde{R})(\mathbb{E}_{\tilde{R}}^{\sigma_2}f_2)^{p_2}\bigg)^{p/{p_2}}\nonumber\\
&\le& [\vec{w}]_{A_{\vec{P}}}
  [v_{\vec{w}}]_{A_\infty}^{p/{p'}}([\sigma_1]_{A_\infty}^{1/{p_1}}+[\sigma_2]_{A_\infty}^{1/{p_2}})^p\|f_1\|_{L^{p_1}(\sigma_1)}^p
  \cdot\|f_2\|_{L^{p_2}(\sigma_2)}^p,\label{eq:e23}
\end{eqnarray}
where (\ref{eq:e10}) is used in the last step.

Combining (\ref{eq:e21}), (\ref{eq:e22}) and (\ref{eq:e23}), we get
\begin{eqnarray*}
 I_2&\lesssim& [\vec{w}]_{A_{\vec{P}}}\bigg([\sigma_1]_{A_\infty}^{p/{p_1}}
  [\sigma_2]_{A_\infty}^{p/{p_2}}+[v_{\vec{w}}]_{A_\infty}^{p/{p'}}
  ([\sigma_1]_{A_\infty}^{p/{p_1}}+
  [\sigma_2]_{A_\infty}^{p/{p_2}})\bigg)\\
  &&\quad\quad\times\|f_1\|_{L^{p_1}(\sigma_1)}^p\cdot
  \|f_2\|_{L^{p_2}(\sigma_2)}^p.
\end{eqnarray*}
By symmetry, we also have
\begin{eqnarray*}
 I_3&\lesssim& [\vec{w}]_{A_{\vec{P}}}\bigg([\sigma_1]_{A_\infty}^{p/{p_1}}
  [\sigma_2]_{A_\infty}^{p/{p_2}}+[v_{\vec{w}}]_{A_\infty}^{p/{p'}}
  ([\sigma_1]_{A_\infty}^{p/{p_1}}+
  [\sigma_2]_{A_\infty}^{p/{p_2}})\bigg)\\
  &&\quad\quad\times\|f_1\|_{L^{p_1}(\sigma_1)}^p\cdot
  \|f_2\|_{L^{p_2}(\sigma_2)}^p.
\end{eqnarray*}

\subsection{Estimate of $I_4$}
Similarly to the previous arguments, we have
\begin{eqnarray*}
&&\int_{E_l(Q)}A^{}_{\mathscr{D},\mathcal{S}}(f_1\sigma_1 1_{Q\cap \Omega_{l+2}}, f_2\sigma_2 1_{Q\cap \Omega_{l+2}})(x) v_{\vec{w}}dx \\
&=&\sum_{R\in\mathcal{Q}_{l+2}\atop R\subset Q}\sum_{\tilde{R}\in \mathcal{Q}_{l+2}\atop \tilde{R}\subset Q}\int_{\tilde{R}}A^{}_{\mathscr{D},\mathcal{S}}(v_{\vec{w}} 1_{E_l(Q)}, \sigma_1 1_R)(x)\sigma_2dx
\cdot\mathbb{E}_R^{\sigma_1}f_1\cdot\mathbb{E}_{\tilde{R}}^{\sigma_2}f_2\\
&\le&\sum_{R\in\mathcal{Q}_{l+2}, R\subset Q\atop \mathbb{E}_R^{\sigma_1}f_1\le 16 \mathbb{E}_{\Gamma(Q)}^{\sigma_1}f_1}\sum_{\tilde{R}\in \mathcal{Q}_{l+2}, \tilde{R}\subset Q\atop \mathbb{E}_{\tilde{R}}^{\sigma_2}f_2\le 16 \mathbb{E}_{\tilde{\Gamma}(Q)}^{\sigma_2}f_2}\int_{\tilde{R}}A^{}_{\mathscr{D},\mathcal{S}}(v_{\vec{w}} 1_{E_l(Q)}, \sigma_1 1_R)(x)\sigma_2dx\\
&&\qquad\qquad\times\mathbb{E}_R^{\sigma_1}f_1\cdot\mathbb{E}_{\tilde{R}}^{\sigma_2}f_2\\
&&\quad+\sum_{R\in\mathcal{Q}_{l+2}, R\subset Q\atop \mathbb{E}_R^{\sigma_1}f_1\le 16 \mathbb{E}_{\Gamma(Q)}^{\sigma_1}f_1}\sum_{\tilde{R}\in \mathcal{Q}_{l+2}, \tilde{R}\subset Q\atop \mathbb{E}_{\tilde{R}}^{\sigma_2}f_2> 16 \mathbb{E}_{\tilde{\Gamma}(Q)}^{\sigma_2}f_2}\int_{\tilde{R}}A^{}_{\mathscr{D},\mathcal{S}}(v_{\vec{w}} 1_{E_l(Q)}, \sigma_1 1_R)(x)\sigma_2dx\\
&&\qquad\qquad\times\mathbb{E}_R^{\sigma_1}f_1\cdot\mathbb{E}_{\tilde{R}}^{\sigma_2}f_2\\
&&\quad+\sum_{R\in\mathcal{Q}_{l+2}, R\subset Q\atop \mathbb{E}_R^{\sigma_1}f_1> 16 \mathbb{E}_{\Gamma(Q)}^{\sigma_1}f_1}\sum_{\tilde{R}\in \mathcal{Q}_{l+2}, \tilde{R}\subset Q\atop \mathbb{E}_{\tilde{R}}^{\sigma_2}f_2\le 16 \mathbb{E}_{\tilde{\Gamma}(Q)}^{\sigma_2}f_2}\int_{\tilde{R}}A^{}_{\mathscr{D},\mathcal{S}}(v_{\vec{w}} 1_{E_l(Q)}, \sigma_1 1_R)(x)\sigma_2dx\\
&&\qquad\qquad\times\mathbb{E}_R^{\sigma_1}f_1\cdot\mathbb{E}_{\tilde{R}}^{\sigma_2}f_2\\
&&\quad+\sum_{R\in\mathcal{Q}_{l+2}, R\subset Q\atop \mathbb{E}_R^{\sigma_1}f_1> 16 \mathbb{E}_{\Gamma(Q)}^{\sigma_1}f_1}\sum_{\tilde{R}\in \mathcal{Q}_{l+2}, \tilde{R}\subset Q\atop \mathbb{E}_{\tilde{R}}^{\sigma_2}f_2> 16 \mathbb{E}_{\tilde{\Gamma}(Q)}^{\sigma_2}f_2}\int_{\tilde{R}}A^{}_{\mathscr{D},\mathcal{S}}(v_{\vec{w}} 1_{E_l(Q)}, \sigma_1 1_R)(x)\sigma_2dx\\
&&\qquad\qquad\times\mathbb{E}_R^{\sigma_1}f_1\cdot\mathbb{E}_{\tilde{R}}^{\sigma_2}f_2\\
&:=& I_{41}(l,Q)+I_{42}(l,Q)+I_{43}(l,Q)+I_{44}(l,Q).
\end{eqnarray*}

\subsubsection{Estimate of $I_{41}$}\label{sec:s311}
We have
\begin{eqnarray*}
&&\sum_{R\in\mathcal{Q}_{l+2}, R\subset Q\atop \mathbb{E}_R^{\sigma_1}
     f_1\le 16 \mathbb{E}_{\Gamma(Q)}^{\sigma_1}f_1}
  \sum_{\tilde{R}\in \mathcal{Q}_{l+2}, \tilde{R}\subset Q\atop
     \mathbb{E}_{\tilde{R}}^{\sigma_2}f_2\le 16 \mathbb{E}_{\tilde{\Gamma}(Q)}^{\sigma_2}f_2}
  \int_{\tilde{R}}A^{}_{\mathscr{D},\mathcal{S}}(v_{\vec{w}} 1_{E_l(Q)}, \sigma_1 1_R)(x)\sigma_2dx \\
&& \qquad\qquad\qquad \qquad \times
      \mathbb{E}_R^{\sigma_1}f_1\cdot\mathbb{E}_{\tilde{R}}^{\sigma_2}f_2\\
&\le&16^2 \int_{Q\cap \Omega_{l+2}}A^{}_{\mathscr{D},\mathcal{S}}(v_{\vec{w}} 1_{E_l(Q)}, \sigma_1 1_{Q\cap \Omega_{l+2}})(x)\sigma_2dx
\cdot\mathbb{E}_{\Gamma(Q)}^{\sigma_1}f_1\cdot\mathbb{E}_{\tilde{\Gamma}(Q)}^{\sigma_2}f_2\\
&\le&16^2 \int_{E_l(Q)}A^{}_{\mathscr{D},\mathcal{S}}(\sigma_1 1_Q, \sigma_2 1_Q)(x)v_{\vec{w}}dx
\cdot\mathbb{E}_{\Gamma(Q)}^{\sigma_1}f_1\cdot\mathbb{E}_{\tilde{\Gamma}(Q)}^{\sigma_2}f_2.
\end{eqnarray*}

Similarly to (\ref{eq:e20}), for $x\in E_l(Q)$, we have
\[
  A^{}_{\mathscr{D},\mathcal{S}}(\sigma_1 1_{Q}, \sigma_2 1_{Q})(x)
  \le 2A^{}_{\mathscr{D},\mathcal{S}_l(Q)}(\sigma_1 1_{Q}, \sigma_2 1_{Q})(x).
\]

By Lemma~\ref{lm:l1}, we have
\begin{eqnarray*}
&&\sum_{l\in\bbZ}\sum_{Q\in\mathcal{Q}_l}v_{\vec{w}}(E_l(Q))^{1-p}I_{41}(l,Q)^p\\
&\le&\sum_{l\in\bbZ}\sum_{Q\in\mathcal{Q}_l}v_{\vec{w}}(E_l(Q))^{1-p}\bigg( \int_{E_l(Q)}A^{}_{\mathscr{D},\mathcal{S}}(\sigma_1 1_Q, \sigma_2 1_Q)(x)v_{\vec{w}}dx\bigg)^p\\&&\quad
\cdot(\mathbb{E}_{\Gamma(Q)}^{\sigma_1}f_1\cdot\mathbb{E}_{\tilde{\Gamma}(Q)}^{\sigma_2}f_2)^p\\
&\lesssim&
\sum_{l\in\bbZ}\sum_{Q\in\mathcal{Q}_l}v_{\vec{w}}(E_l(Q))^{1-p}\bigg( \int_{E_l(Q)}A^{}_{\mathscr{D},\mathcal{S}_l(Q)}(\sigma_1 1_Q, \sigma_2 1_Q)(x)v_{\vec{w}}dx\bigg)^p\\&&\quad
\cdot(\mathbb{E}_{\Gamma(Q)}^{\sigma_1}f_1\cdot\mathbb{E}_{\tilde{\Gamma}(Q)}^{\sigma_2}f_2)^p\\
&\le&
  \sum_{l\in\bbZ}\sum_{Q\in\mathcal{Q}_l} \int_{E_l(Q)}(A^{}_{\mathscr{D},\mathcal{S}_l(Q)}(\sigma_1 1_{Q}, \sigma_2 1_{Q}))^p v_{\vec{w}}dx
\cdot(\mathbb{E}_{\Gamma(Q)}^{\sigma_1}f_1\cdot\mathbb{E}_{\tilde{\Gamma}(Q)}^{\sigma_2}f_2)^p\\
&\lesssim&[\vec{w}]_{A_{\vec{p}}}\bigg( \sum_{G\in\mathcal{G}}\sum_{l\in\bbZ, Q\in\mathcal{Q}_l\atop \Gamma(Q)=G}\sum_{Q_{l,\eta}\in \mathcal{S}_l(Q)}\sigma_1(Q_{l,\eta})(\mathbb{E}_{G}^{\sigma_1}f_1)^{p_1}\bigg)^{p/{p_1}}\\
   &&\quad\cdot\bigg( \sum_{\tilde{G}\in\tilde{\mathcal{G}}}\sum_{l\in\bbZ, Q\in\mathcal{Q}_l\atop \tilde{\Gamma}(Q)=\tilde{G}}\sum_{Q_{l,\eta}\in \mathcal{S}_l(Q)}\sigma_2(Q_{l,\eta})(\mathbb{E}_{\tilde{G}}^{\sigma_2}f_2)^{p_2}\bigg)^{p/{p_2}}\\
&&\hskip 40mm \mbox{(by Lemma~\ref{lm:l1} and H\"older's inequality)}\\
   &\le&[\vec{w}]_{A_{\vec{p}}}[\sigma_1]_{A_\infty}^{p/{p_1}}[\sigma_2]_{A_\infty}^{p/{p_2}}\bigg( \sum_{G\in\mathcal{G}}
   \sigma_1(G)(\mathbb{E}_{G}^{\sigma_1}f_1)^{p_1}\bigg)^{p/{p_1}} \\
&&\qquad \times
    \bigg( \sum_{\tilde{G}\in\tilde{\mathcal{G}}}
      \sigma_2(\tilde{G})(\mathbb{E}_{\tilde{G}}^{\sigma_2}f_2)^{p_2}\bigg)^{p/{p_2}}\quad\mbox{(by Lemma~\ref{lm:l0})}\\
   &\lesssim&[\vec{w}]_{A_{\vec{p}}}[\sigma_1]_{A_\infty}^{p/{p_1}}[\sigma_2]_{A_\infty}^{p/{p_2}}\|f_1\|_{L^{p_1}(\sigma_1)}^{p}
   \|f_2\|_{L^{p_2}(\sigma_2)}^{p},
\end{eqnarray*}
where (\ref{eq:e14}) and (\ref{eq:e15}) are used in the last step.

\subsubsection{Estimates of $I_{42}$ and $I_{43}$}
We have
\begin{eqnarray*}
&&\hskip -15pt\sum_{R\in\mathcal{Q}_{l+2}, R\subset Q\atop \mathbb{E}_R^{\sigma_1}f_1\le 16 \mathbb{E}_{\Gamma(Q)}^{\sigma_1}f_1}
  \sum_{\tilde{R}\in \mathcal{Q}_{l+2}, \tilde{R}\subset Q\atop
     \mathbb{E}_{\tilde{R}}^{\sigma_2}f_2> 16 \mathbb{E}_{\tilde{\Gamma}(Q)}^{\sigma_2}f_2}
\hskip -10pt\int_{\tilde{R}}
  \! A^{}_{\mathscr{D},\mathcal{S}}(v_{\vec{w}} 1_{E_l(Q)}, \sigma_1 1_R)(x)\sigma_2dx
\cdot\mathbb{E}_R^{\sigma_1}f_1\cdot\mathbb{E}_{\tilde{R}}^{\sigma_2}f_2\\
&\le&\sum_{\tilde{R}\in \mathcal{Q}_{l+2}, \tilde{R}\subset Q\atop \mathbb{E}_{\tilde{R}}^{\sigma_2}f_2> 16 \mathbb{E}_{\tilde{\Gamma}(Q)}^{\sigma_2}f_2}\int_{\tilde{R}}A^{}_{\mathscr{D},\mathcal{S}}(v_{\vec{w}} 1_{E_l(Q)}, \sigma_1 1_Q)(x)\sigma_2dx
\cdot\mathbb{E}_{\Gamma(Q)}^{\sigma_1}f_1\cdot\mathbb{E}_{\tilde{R}}^{\sigma_2}f_2.
\end{eqnarray*}
For simplicity, set
\[
  g_Q:=\sum_{\tilde{R}\in \mathcal{Q}_{l+2}, \tilde{R}\subset Q\atop \mathbb{E}_{\tilde{R}}^{\sigma_2}f_2> 16 \mathbb{E}_{\tilde{\Gamma}(Q)}^{\sigma_2}f_2}
  \mathbb{E}_{\tilde{R}}^{\sigma_2}f_2 1_{\tilde{R}}.
\]
We have
\begin{eqnarray*}
&&\hskip -25pt\sum_{l\in\bbZ, Q\in\mathcal{Q}_l}
 \!\! v_{\vec{w}}(E_l(Q))^{1-p}
  \bigg(\int_{E_l(Q)}\!A^{}_{\mathscr{D},\mathcal{S}}(\sigma_1 1_Q, g_Q\sigma_2)(x)v_{\vec{w}}dx
\cdot\mathbb{E}_{\Gamma(Q)}^{\sigma_1}f_1\bigg)^p\\
&=&\sum_{G\in\mathcal{G}}\sum_{l\in\bbZ, Q\in\mathcal{Q}_l\atop \Gamma(Q)=G}\!\!
  v_{\vec{w}}(E_l(Q))^{1-p}
\bigg(\int_{E_l(Q)}\! A^{}_{\mathscr{D},\mathcal{S}}(\sigma_1 1_Q, g_Q\sigma_2)(x)v_{\vec{w}}dx
\cdot\mathbb{E}_{G}^{\sigma_1}f_1\bigg)^p.
\end{eqnarray*}
Set
\[
  h_{G}:=\sup_{Q: Q\in\mathcal{Q}_l,l\in\bbZ \atop \Gamma(Q)=G}g_Q.
\]
We have
\begin{eqnarray*}
\sum_{G\in \mathcal{G}}\|h_G\|_{L^{p_2}(\sigma_2)}^{p_2}
&=&\sum_{G\in \mathcal{G}}\int h_{G}^{p_2}\sigma_2\\
&\le& \sum_{G\in \mathcal{G}}\sum_{Q: Q\in\mathcal{Q}_l,l\in\bbZ \atop \Gamma(Q)=G}\int g_Q^{p_2}\sigma_2\\
&=&\sum_{G\in \mathcal{G}}\sum_{Q: Q\in\mathcal{Q}_l,l\in\bbZ \atop \Gamma(Q)=G}\sum_{\tilde{R}\in \mathcal{Q}_{l+2}, \tilde{R}\subset Q\atop \mathbb{E}_{\tilde{R}}^{\sigma_2}f_2> 16 \mathbb{E}_{\tilde{\Gamma}(Q)}^{\sigma_2}f_2}
  (\mathbb{E}_{\tilde{R}}^{\sigma_2}f_2)^{p_2}\sigma_2(\tilde{R})\\
  &=&\sum_{l\in\bbZ, Q\in\mathcal{Q}_l}\sum_{\tilde{R}\in \mathcal{Q}_{l+2}, \tilde{R}\subset Q\atop \mathbb{E}_{\tilde{R}}^{\sigma_2}f_2> 16 \mathbb{E}_{\tilde{\Gamma}(Q)}^{\sigma_2}f_2}
  (\mathbb{E}_{\tilde{R}}^{\sigma_2}f_2)^{p_2}\sigma_2(\tilde{R})\\
  &\le&\|f_2\|_{L^{p_2}(\sigma_2)}^{p_2}. \quad\mbox{(by (\ref{eq:e10}))}
\end{eqnarray*}
Since $E_l(Q)$ are disjoint,  we have
\begin{eqnarray*}
&&\sum_{l\in\bbZ}\sum_{Q\in\mathcal{Q}_l}v_{\vec{w}}(E_l(Q))^{1-p}I_{42}(l,Q)^p\\
&\le&\sum_{G\in\mathcal{G}}\sum_{l\in\bbZ, Q\in\mathcal{Q}_l\atop \Gamma(Q)=G}
  \!v_{\vec{w}}(E_l(Q))^{1-p}
  \bigg(\int_{E_l(Q)}\!
  A^{}_{\mathscr{D},\mathcal{S}}(\sigma_1 1_Q, g_Q\sigma_2)(x)v_{\vec{w}}dx
\cdot\mathbb{E}_{G}^{\sigma_1}f_1\bigg)^p\\
&\le& \sum_{G\in\mathcal{G}}\sum_{l\in\bbZ, Q\in\mathcal{Q}_l\atop \Gamma(Q)=G}
\int_{E_l(Q)}(A^{}_{\mathscr{D},\mathcal{S}}(\sigma_1 1_G, h_G\sigma_2))^pv_{\vec{w}}dx
\cdot(\mathbb{E}_{G}^{\sigma_1}f_1)^p\\
&\le&\sum_{G\in\mathcal{G}}\int_G (A^{}_{\mathscr{D},\mathcal{S}}(\sigma_1 1_G, h_G\sigma_2))^p v_{\vec{w}}dx\cdot (\mathbb{E}_{G}^{\sigma_1}f_1)^p\\
&\lesssim& [\vec{w}]_{A_{\vec{P}}}[\sigma_1]_{A_\infty}^{p/{p_1}}([v_{\vec{w}}]_{A_\infty}^{1/{p'}}+[\sigma_2]_{A_\infty}^{1/{p_2}})^p
\sum_{G\in\mathcal{G}}\sigma_1(G)^{p/{p_1}}(\mathbb{E}_{G}^{\sigma_1}f_1)^p\|h_G\|_{L^{p_2}(\sigma_2)}^p\\
&&\hskip 60mm \mbox{(by Lemma~\ref{lm:l2})}\\
&\le&[\vec{w}]_{A_{\vec{P}}}[\sigma_1]_{A_\infty}^{p/{p_1}}([v_{\vec{w}}]_{A_\infty}^{1/{p'}}+[\sigma_2]_{A_\infty}^{1/{p_2}})^p\|f_1\|_{L^{p_1}(\sigma_1)}^p \bigg(\sum_{G\in\mathcal{G}}\|h_G\|_{L^{p_2}(\sigma_2)}^{p_2}\bigg)^{p/{p_2}}\\
&&\hskip 40mm \mbox{(by H\"older's inequality and (\ref{eq:e14}))}\\
&\le&[\vec{w}]_{A_{\vec{P}}}[\sigma_1]_{A_\infty}^{p/{p_1}}([v_{\vec{w}}]_{A_\infty}^{1/{p'}}+[\sigma_2]_{A_\infty}^{1/{p_2}})^p\|f_1\|_{L^{p_1}(\sigma_1)}^p
\cdot\|f_2\|_{L^{p_2}(\sigma_2)}^p.
\end{eqnarray*}
By symmetry, we get
\begin{eqnarray*}
\sum_{l\in\bbZ}\sum_{Q\in\mathcal{Q}_l}v_{\vec{w}}(E_l(Q))^{1-p}I_{43}(l,Q)^p
&\lesssim&[\vec{w}]_{A_{\vec{P}}}[\sigma_2]_{A_\infty}^{p/{p_2}}([v_{\vec{w}}]_{A_\infty}^{1/{p'}}+[\sigma_1]_{A_\infty}^{1/{p_1}})^p
\\&&\quad\quad\times\|f_1\|_{L^{p_1}(\sigma_1)}^p
\cdot\|f_2\|_{L^{p_2}(\sigma_2)}^p.
\end{eqnarray*}

\subsubsection{Estimate of $I_{44}$}
We have
\begin{eqnarray*}
&&\hskip -1.6em\sum_{R\in\mathcal{Q}_{l+2}, R\subset Q\atop \mathbb{E}_R^{\sigma_1}f_1> 16 \mathbb{E}_{\Gamma(Q)}^{\sigma_1}f_1}
 \sum_{\tilde{R}\in \mathcal{Q}_{l+2}, \tilde{R}\subset Q\atop \mathbb{E}_{\tilde{R}}^{\sigma_2}f_2> 16 \mathbb{E}_{\tilde{\Gamma}(Q)}^{\sigma_2}f_2}
 \hskip -0.7em \int_{\tilde{R}}A^{}_{\mathscr{D},\mathcal{S}}(v_{\vec{w}} 1_{E_l(Q)}, \sigma_1 1_R)(x)\sigma_2dx
\cdot\mathbb{E}_R^{\sigma_1}f_1\cdot\mathbb{E}_{\tilde{R}}^{\sigma_2}f_2\\
&\le&\bigg(\!\!\sum_{R\in \mathcal{Q}_{l+2}, R\subset Q\atop \mathbb{E}_{R}^{\sigma_1}f_1> 16 \mathbb{E}_{\Gamma(Q)}^{\sigma_1}f_1}\!\!
 \sigma_1(R)^{-p_1'/{p_1}}\bigg(\int_{R}A^{}_{\mathscr{D},\mathcal{S}}(v_{\vec{w}} 1_{E_l(Q)}, g_Q\sigma_2 )(x)\sigma_1dx\bigg)^{p_1'}\bigg)^{1/{p_1'}}\\
 &&\quad\cdot \bigg(\sum_{R\in \mathcal{Q}_{l+2}, R\subset Q\atop \mathbb{E}_{R}^{\sigma_1}f_1> 16 \mathbb{E}_{\Gamma(Q)}^{\sigma_1}f_1}
 \sigma_1(R)(\mathbb{E}_{R}^{\sigma_1}f_1)^{p_1}\bigg)^{1/{p_1}}\\
 &\le&\!\bigg(\!\int_{Q}(A^{}_{\mathscr{D},\mathcal{S}}(v_{\vec{w}} 1_{E_l(Q)}, g_Q\sigma_2 ))^{p_1'}\sigma_1dx\!\bigg)^{1/{p_1'}}
  \bigg(\hskip -0.8em \sum_{R\in \mathcal{Q}_{l+2},
    R\subset Q\atop \mathbb{E}_{R}^{\sigma_1}f_1> 16 \mathbb{E}_{\Gamma(Q)}^{\sigma_1}f_1}
   \hskip -1.5em \sigma_1(R)(\mathbb{E}_{R}^{\sigma_1}f_1)^{p_1}\!\!\bigg)^{1/{p_1}}\\
 &\lesssim&[\vec{w}]_{A_{\vec{P}}}^{1/p}
  [v_{\vec{w}}]_{A_\infty}^{1/{p'}}([\sigma_1]_{A_\infty}^{1/{p_1}}+[\sigma_2]_{A_\infty}^{1/{p_2}}) v_{\vec{w}}(Q)^{1/{p'}}\|g_Q\|_{L^{p_2}(\sigma_2)}\\
  &&\quad\cdot \bigg(\sum_{R\in \mathcal{Q}_{l+2}, R\subset Q\atop \mathbb{E}_{R}^{\sigma_1}f_1> 16 \mathbb{E}_{\Gamma(Q)}^{\sigma_1}f_1}
 \sigma_1(R)(\mathbb{E}_{R}^{\sigma_1}f_1)^{p_1}\bigg)^{1/{p_1}},
\end{eqnarray*}
where Lemma~\ref{lm:l2} and Lemma~\ref{lm:w} are used in the last step.
It follows that
\begin{eqnarray*}
&&\sum_{l\in\bbZ}\sum_{Q\in\mathcal{Q}_l}v_{\vec{w}}(E_l(Q))^{1-p}I_{44}(l,Q)^p\\
&=&\sum_{l\in\bbZ}\sum_{Q\in\mathcal{Q}_l}v_{\vec{w}}(E_l(Q))^{1-p}\bigg( \sum_{R\in\mathcal{Q}_{l+2}, R\subset Q\atop \mathbb{E}_R^{\sigma_1}f_1> 16 \mathbb{E}_{\Gamma(Q)}^{\sigma_1}f_1}\sum_{\tilde{R}\in \mathcal{Q}_{l+2}, \tilde{R}\subset Q\atop \mathbb{E}_{\tilde{R}}^{\sigma_2}f_2> 16 \mathbb{E}_{\tilde{\Gamma}(Q)}^{\sigma_2}f_2} \mathbb{E}_{R}^{\sigma_1}f_1\cdot\mathbb{E}_{\tilde{R}}^{\sigma_2}f_2\\
&&\quad \int_{\tilde{R}}A^{}_{\mathscr{D},\mathcal{S}}(v_{\vec{w}} 1_{E_l(Q)}, \sigma_1 1_R)(x)\sigma_2dx \bigg)^p\\
&\lesssim&[\vec{w}]_{A_{\vec{P}}}
  [v_{\vec{w}}]_{A_\infty}^{p/{p'}}([\sigma_1]_{A_\infty}^{1/{p_1}}+[\sigma_2]_{A_\infty}^{1/{p_2}})^p
\bigg(\sum_{l\in\bbZ}\sum_{Q\in\mathcal{Q}_l}\|g_Q\|_{L^{p_2}(\sigma_2)}^{p_2}\bigg)^{p/{p_2}}\\
&&\quad\cdot \bigg(\sum_{l\in\bbZ}\sum_{Q\in\mathcal{Q}_l}\sum_{R\in \mathcal{Q}_{l+2}, R\subset Q\atop \mathbb{E}_{R}^{\sigma_1}f_1> 16 \mathbb{E}_{\Gamma(Q)}^{\sigma_1}f_1}
 \sigma_1(R)(\mathbb{E}_{R}^{\sigma_1}f_1)^{p_1}\bigg)^{p/{p_1}}\\
 &=&[\vec{w}]_{A_{\vec{P}}}
  [v_{\vec{w}}]_{A_\infty}^{p/{p'}}([\sigma_1]_{A_\infty}^{1/{p_1}}+[\sigma_2]_{A_\infty}^{1/{p_2}})^p\\
&&\qquad \times
  \bigg(\sum_{l\in\bbZ, Q\in\mathcal{Q}_l}\sum_{\tilde{R}\in\mathcal{Q}_{l+2},\tilde{R}\subset Q\atop \mathbb{E}_{\tilde{R}}^{\sigma_2}f_2> 16 \mathbb{E}_{\tilde{\Gamma}(Q)}^{\sigma_2}f_2}\sigma_2(\tilde{R})(\mathbb{E}_{\tilde{R}}^{\sigma_2}f_2)^{p_2}\bigg)^{p/{p_2}}\\
&&\qquad\times \bigg(\sum_{l\in\bbZ}\sum_{Q\in\mathcal{Q}_l}\sum_{R\in \mathcal{Q}_{l+2}, R\subset Q\atop \mathbb{E}_{R}^{\sigma_1}f_1> 16 \mathbb{E}_{\Gamma(Q)}^{\sigma_1}f_1}
 \sigma_1(R)(\mathbb{E}_{R}^{\sigma_1}f_1)^{p_1}\bigg)^{p/{p_1}}\\
 &\lesssim&[\vec{w}]_{A_{\vec{P}}}
  [v_{\vec{w}}]_{A_\infty}^{p/{p'}}([\sigma_1]_{A_\infty}^{1/{p_1}}+[\sigma_2]_{A_\infty}^{1/{p_2}})^p
  \|f_1\|_{L^{p_1}(\sigma_1)}^p\cdot\|f_2\|_{L^{p_2}(\sigma_2)}^p,
\end{eqnarray*}
where  (\ref{eq:e9}) and (\ref{eq:e10}) are used in the last step.

Summing up the above arguments, we get
\begin{eqnarray*}
 I_4&\lesssim& [\vec{w}]_{A_{\vec{P}}}\bigg([\sigma_1]_{A_\infty}^{p/{p_1}}
  [\sigma_2]_{A_\infty}^{p/{p_2}}+[v_{\vec{w}}]_{A_\infty}^{p/{p'}}
  ([\sigma_1]_{A_\infty}^{p/{p_1}}+
  [\sigma_2]_{A_\infty}^{p/{p_2}})\bigg)\\
  &&\quad\quad\times\|f_1\|_{L^{p_1}(\sigma_1)}^p\cdot
  \|f_2\|_{L^{p_2}(\sigma_2)}^p.
\end{eqnarray*}
This completes the proof of (\ref{eq:main:1}).

\subsection{Sharpness of the strong type estimates }
Finally, we prove the sharpness.
We use the example  in \cite{LS}. That is,
\[
 R_1(\vec{f})(x)\!=p.v.\!\int_{(\bbR^n)^m}\frac{\sum_{j=1}^m(x_1-(y_j)_1)}{(\sum_{j=1}^m|x-y_j|^2)^{(nm+1)/2}}f_1(y_1)\cdots f_m(y_m)dy_1\cdots dy_m.
\]
Assume that $0<\varepsilon<1$. Let
\[
  f_i(x)=|x|^{\varepsilon-n}\chi^{}_{(0,1]^n}(x)\quad\mbox{and}\quad w_i(x)=|x|^{(n-\varepsilon)(p_i-1)}, \quad i=1, \cdots, m.
\]
Then we have $v_{\vec{w}}=|x|^{(n-\varepsilon)(mp-1)}$,
\[
  [\vec{w}]_{A_{\vec{P}}}=[v_{\vec{w}}]_{A_{mp}}\approx (1/\varepsilon)^{mp-1}
   \quad\mbox{and}\quad [\sigma_i]_{A_\infty}\lesssim 1/{\varepsilon}\quad\mbox{$i=1,\cdots,m$}.
\]
Moreover,
\[
  \|R_1(\vec{f})\|_{L^p(v_{\vec{w}})}\ge (1/\varepsilon)^{m+1/p}
   \quad\mbox{and}\quad \prod_{i=1}^m \|f_i\|_{L^{p_i}(w_i)}\approx (1/\varepsilon)^{1/p}.
\]
It follows that our result is sharp whenever $\max_i\{p_i\}\le p'/{(mp-1)}$.

This completes the proof of Theorem~\ref{thm:main}.

\section{Proof of Theorem~\ref{thm:m1}}

In this section, we give the proof  of Theorem~\ref{thm:m1}.
We begin with the Whitney decomposition.
\begin{Proposition}\cite[Proposition 7.3.4]{Gra}\label{prop:whitney}
Let $\Omega$ be an open nonempty proper
subset of $\bbR^n$. Then there exists a family of closed cubes $\{Q_j\}_j$ such that
\begin{enumerate}
\item $\bigcup_j Q_j=\Omega$ and the $Q_j$'s have disjoint interiors;
\item $\sqrt{n}l(Q_j)\le \dist(Q_j, \Omega^c)\le 4\sqrt{n}l(Q_j)$;
\item if the boundaries of two cubes $Q_j$ and $Q_k$ touch, then
\[
  \frac{1}{4}\le \frac{l(Q_j)}{l(Q_k)}\le 4;
\]
\item  there exists some constant $1<\gamma<5/4$ such that $\sum_{j}\chi^{}_{\gamma Q_j}(x)\le C_n$.
\end{enumerate}
\end{Proposition}

Next we give a weak type estimate for the multilinear maximal function. Recall that
the multilinear maximal function is defined by
\[
  \mathcal{M}(\vec{f})=\sup_{Q\ni x}\prod_{i=1}^m\frac{1}{|Q|}\int_{Q}|f_i(y_i)|dy_i
\]
and the dyadic maximal function is defined by
\[
   \mathcal{M}^{\mathscr{D}}(\vec{f})(x)=\sup_{Q\ni x, Q\in\mathscr{D}
    }\prod_{i=1}^m\frac{1}{|Q|}\int_{Q}|f_i(y_i)|dy_i.
\]
\begin{Lemma}\label{lm:mweak}
Let $\vec{P}=(p_1,\cdots,p_m)$ with $1/p=1/{p_1}+\cdots+1/{p_m}$
and $1< p_1,\cdots,p_m<\infty$. Suppose that $\vec{w}=(w_1,\cdots,w_m)$ with $\vec{w}\in A_{\vec{P}}$.
Then
\[
\|\mathcal{M}(\vec{f})\|_{L^{p,\infty}(v_{\vec{w}})}\le C_{m,n,\vec{P}}[\vec{w}]_{A_{\vec{P}}}^{1/p}\prod_{i=1}^m\|f_i\|_{L^{p_i}(w_i)}.
\]
\end{Lemma}
\begin{proof}
 In \cite{DLP}, the authors proved that
there exists $2^n$ family of dyadic grids $\mathscr{D}_\beta$ such that
\[
  \mathcal{M}(\vec{f})(x)\le 6^{mn}\sum_{\beta=1}^{2^n}\mathcal{M}^{\mathscr{D}_\beta}(\vec{f})(x),
\]
where
\[
    \mathcal{M}^{\mathscr{D}_\beta}(\vec{f})(x)=\sup_{Q\ni x, Q\in\mathscr{D}_\beta
    }\prod_{i=1}^m\frac{1}{|Q|}\int_{Q}|f_i(y_i)|dy_i.
\]
For some fixed dyadic grid $\mathscr{D}$,
\[
  \{x\in\bbR^n: \mathcal{M}^{\mathscr{D}}(\vec{f})>\alpha\}=\bigcup_k Q_k,
\]
where $\{Q_k\}_k$ are disjoint dyadic cubes in $\mathscr{D}$ and
\[
  \prod_{i=1}^m\frac{1}{|Q_k|}\int_{Q_k}|f_i(y_i)|dy_i>\alpha.
\]
It follows that
\begin{eqnarray*}
\alpha^p \bigg(\sum_{k}v_{\vec{w}}(Q_k)\bigg)
&\le&  \sum_k \bigg( \prod_{i=1}^m \frac{1}{|Q_k|}\int_{Q_k} |f_i|dy_i\bigg)^p v_{\vec{w}}(Q_k)  \\
&\le&  \sum_k \bigg( \prod_{i=1}^m \int_{Q_k} |f_i|^{p_i}w_i dy_i\bigg)^{p/{p_i}}
   \frac{v_{\vec{w}}(Q_k)\prod_{i=1}^m\sigma_i(Q_k)^{p/{p_i'}}}{|Q_k|^{mp}}  \\
&\le& [\vec{w}]_{A_{\vec{P}}}
   \sum_k \bigg( \prod_{i=1}^m \int_{Q_k} |f_i|^{p_i}w_i dy_i\bigg)^{p/{p_i}}  \\
&\le& [\vec{w}]_{A_{\vec{P}}} \prod_{i=1}^m\bigg(\sum_k   \int_{Q_k} |f_i|^{p_i}w_i dy_i\bigg)^{p/{p_i}}\\
&\le& [\vec{w}]_{A_{\vec{P}}} \prod_{i=1}^m \|f_i\|_{L^{p_i}(w_i)}^p.
\end{eqnarray*}
Hence
\begin{eqnarray*}
 \|\mathcal{M}^{\mathscr{D}}(\vec{f})\|_{L^{p,\infty}(v_{\vec{w}})}&=&\sup_{\alpha>0}
  \alpha v_{\vec{w}} (\{x\in\bbR^n: \mathcal{M}^{\mathscr{D}}(\vec{f})>\alpha\} )^{1/{p}}\\
  &=& \sup_{\alpha>0} \alpha \bigg(\sum_{k}v_{\vec{w}}(Q_k)\bigg)^{1/{p}}\\
  &\le& [\vec{w}]_{A_{\vec{P}}}^{1/p}\prod_{i=1}^m \|f_i\|_{L^{p_i}(w_i)}.
\end{eqnarray*}
This completes the proof.
\end{proof}

The following result can be proved similarly to \cite[p. 1240]{LOPTT} and we omit the details.

\begin{Lemma}\label{lm:l10}
Let $T$ be an $m$-linear Calder\'{o}n-Zygmund operator and $Q$ be a cube. Set $Q^*=10\sqrt{n}Q$ and $Q^{**}=10\sqrt{n}Q^*$.
 Suppose that $x,z\in Q$ and $y\in Q^*$. Then
\begin{eqnarray*}
&&|T(f_1,\cdots,f_{i-1},f_i\chi^{}_{(Q^{**})^c},f_{i+1},\cdots,f_m)(x)-T(f_1,\cdots,f_{i-1},f_i\chi^{}_{(Q^{**})^c},\\
&&\quad f_{i+1},\cdots,f_m)(y)|\le C\mathcal{M}(f_1,\cdots,f_m)(z).
\end{eqnarray*}
\end{Lemma}

Next we give a characterization of the weak boundedness of
multilinear Calder\'{o}n-Zygmund operators.
\begin{Lemma}\label{lm:weak}
Let $1<p,p_1,p_2<\infty$ and $\vec{w}\in A_{\vec{P}}$, where $\vec{w}:=(w_1,  w_2)$ and $\vec{P}:=(p_1,  p_2)$ with $1/p=1/{p_1}+1/{p_2}$.
Suppose that $T$ is a multilinear Calder\'{o}n-Zygmund operator. Then the following assertions are equivalent.
\begin{enumerate}
\item $\|T(f_1\sigma_1,f_2\sigma_2)\|_{L^{p,\infty}(v_{\vec{w}})}\le C\prod_{i=1}^2\|f_i\|_{L^{p_i}(\sigma_i)}$;
\item $\int_Q |T(f_1\sigma_1\chi^{}_Q,f_2\sigma_2\chi^{}_{Q})(x)|v_{\vec{w}}(x)dx\le C'\prod_{i=1}^2\|f_i\|_{L^{p_i}(\sigma_i)}v_{\vec{w}}(Q)^{1/{p'}}$ for all cubes $Q\subset \bbR^n$
and all functions $f_i\in L^{p_i}(\sigma_i)$, $i=1,2$.
\end{enumerate}
\end{Lemma}
\begin{proof}
(i)$\Rightarrow$ (ii):
By the weak type boundness of $T$, we have
\begin{eqnarray*}
&&\int_Q |T(f_1\sigma_1\chi^{}_Q,f_2\sigma_2\chi^{}_{Q})(x)|v_{\vec{w}}(x)dx\\
&=&\int_0^\infty v_{\vec{w}}\{x\in Q: |T(f_1\sigma_1\chi^{}_Q, f_2\sigma_2\chi^{}_{Q})(x)|>\lambda\}d\lambda\\
&\le& \int_0^\infty \min\{v_{\vec{w}}(Q), \lambda^{-p}\|T(f_1\sigma_1\chi^{}_Q,f_2\sigma_2\chi^{}_{Q})\|_{L^{p,\infty}(v_{\vec{w}})}^p\}d\lambda\\
&=&p'\|T(f_1\sigma_1\chi^{}_Q,f_2\sigma_2\chi^{}_{Q})\|_{L^{p,\infty}(v_{\vec{w}})}v_{\vec{w}}(Q)^{1/{p'}}\\
&\le&p'C\prod_{i=1}^2\|f_i\|_{L^{p_i}(\sigma_i)}v_{\vec{w}}(Q)^{1/{p'}}.
\end{eqnarray*}

(ii)$\Rightarrow$(i):
Let $\Omega$ be an open set containing $\{x: |T(f_1\sigma_1, f_2\sigma_2)(x)|>\lambda\}$. Form the Whitney decomposition to $\Omega$, we get Whitney cubes $Q_j$. Set
$Q_j^*=10\sqrt{n}Q_j$ and $Q_j^{**}=10\sqrt{n}Q_j^*$. Let
$\gamma$ be defined as that in Proposition~\ref{prop:whitney}.
In the following, we prove that
\begin{eqnarray}
&&v_{\vec{w}}\{x\in \bbR^n: |T(f_1\sigma_1,f_2\sigma_2)(x)|>2\lambda, \mathcal{M}(f_1\sigma_1,f_2\sigma_2)(x)\le \beta \lambda\}\nonumber \\
&\lesssim& \beta v_{\vec{w}}(\Omega)+\mathcal{T}_*^p\beta^{-p}\lambda^{-p}
\prod_{i=1}^m\|f_i\|_{L^{p_i}(\sigma_i)}^p,\label{eq:e1}
\end{eqnarray}
where
\[
  \mathcal{T}_*=\sup_{\|f_i\|_{L^{p_i}(\sigma_i)}\le1\atop i=1,2}\sup_{Q}v_{\vec{w}}(Q)^{-1/{p'}}\int_Q |T(f_1\sigma_1\chi^{}_Q, f_2\sigma_2\chi^{}_{Q})(x)|v_{\vec{w}}(x)dx.
\]

By Whitney's decomposition, we only need to estimate
\begin{equation}\label{eq:q}
v_{\vec{w}}\{x\in Q_j: |T(f_1\sigma_1, f_2\sigma_2)(x)|>2\lambda, \mathcal{M}(f_1\sigma_1, f_2\sigma_2)(x)\le \beta \lambda\}.
\end{equation}
Assume that there exists some $z_j\in Q_j$ such that $\mathcal{M}(f_1\sigma_1, f_2\sigma_2)(z_j)\le \beta \lambda$.
Otherwise, $(\ref{eq:q})$ is zero. By
the property of Whitney decomposition, we can also choose some
 $y_j\in Q_j^*$ such that $y_j\in \Omega^c$. Since $\{x: |T(f_1\sigma_1, f_2\sigma_2)(x)|>\lambda\}\subset \Omega$,
we have $|T(f_1\sigma_1, f_2\sigma_2)(y_j)|\le\lambda$.

For any $f_i$, $i=1,2$, denote $f_i^0=f_i\chi^{}_{\gamma Q_j}$
and  $f_i^\infty=f_i\chi^{}_{(\gamma Q_j)^c}$.
We consider every $f_i^{\alpha_i}$ separately, where $\alpha_i=0$ or $\infty$.

Similarly as that in \cite{LOPTT}, we consider first
the case $\alpha_1= \alpha_2=\infty$.
For $x\in Q_j$, we have
\begin{eqnarray*}
&&|T(f_1^\infty\sigma_1, f_2^\infty\sigma_2)(x)-T(f_1^\infty\sigma_1, f_2^\infty\sigma_2)(y_j)|\\
&\le&|T(f_1\sigma_1\chi^{}_{(Q_j^{**})^c},f_2^\infty\sigma_2)(x)-
T(f_1\sigma_1\chi^{}_{(Q_j^{**})^c},f_2^\infty\sigma_2)(y_j)|\\
&&\quad+|T(f_1\sigma_1\chi^{}_{Q_j^{**}\setminus \gamma Q_j},f_2^\infty\sigma_2)(x)-
T(f_1\sigma_1\chi^{}_{Q_j^{**}\setminus \gamma Q_j},f_2^\infty\sigma_2)(y_j)|\\
&\le&|T(f_1\sigma_1\chi^{}_{(Q_j^{**})^c},f_2^\infty\sigma_2)(x)-
T(f_1\sigma_1\chi^{}_{(Q_j^{**})^c},f_2^\infty\sigma_2)(y_j)|\\
&&\! +|T(f_1\sigma_1\chi^{}_{Q_j^{**}\setminus \gamma Q_j},f_2\sigma_2\chi^{}_{(Q_j^{**})^c}\!)(x)\!-\!
T(f_1\sigma_1\chi^{}_{Q_j^{**}\setminus \gamma Q_j}, f_2\sigma_2\chi^{}_{(Q_j^{**})^c})(y_j)|
\\
&&\quad+
|T(f_1\sigma_1\chi_{Q_j^{**}\setminus \gamma Q_j}, f_2\sigma_2\chi_{Q_j^{**}\setminus \gamma Q_j})(x)|\\
&&\quad+|T(f_1\sigma_1\chi^{}_{Q_j^{**}\setminus \gamma Q_j}, f_2\sigma_2\chi^{}_{Q_j^{**}\setminus \gamma Q_j})(y_j)|\\
&\le&C_1\mathcal{M}(f_1\sigma_1, f_2\sigma_2)(z_j)+|T(f_1\sigma_1\chi^{}_{Q_j^{**}\setminus \gamma Q_j}, f_2\sigma_2\chi^{}_{Q_j^{**}\setminus \gamma Q_j})(y_j)|.
\end{eqnarray*}

Next, suppose that $\alpha_{i}=\infty$ for $i=1$ or $i=2$. Without loss of generality, assume that $i=1$.
  We have
\begin{eqnarray*}
&&|T(f_1^{\infty}\sigma_1, f_2^0\sigma_2)(x)-T(f_1^{\infty}\sigma_1,f_2^0\sigma_2)(y_j)|\\
&\le&|T(f_{1}\sigma_{1}\chi^{}_{(Q_j^{**})^c}, f_{2}^0\sigma_{2} )(x)
-T( f_{1}\sigma_{1}\chi^{}_{(Q_j^{**})^c},f_{2}^0\sigma_{2})(y_j)|\\
&&\quad+
|T( f_{1}\sigma_{1}\chi_{Q_j^{**}\setminus \gamma Q_j},f_{2}^0\sigma_{2})(x)|
+|T( f_{1}\sigma_{1}\chi_{Q_j^{**}\setminus \gamma Q_j},f_{2}^0\sigma_{2})(y_j)|\\
&\le&C_2\mathcal{M}(f_1\sigma_1, f_2\sigma_2)(z_j)
+|T( f_{1}\sigma_{1}\chi_{Q_j^{**}\setminus \gamma Q_j},f_{2}^0\sigma_{2} )(y_j)|.
\end{eqnarray*}
Hence
\begin{eqnarray*}
&&|T(f_1\sigma_1 ,f_2\sigma_2)(x)-T(f_1^0\sigma_1, f_2^0\sigma_2)(x)-T(f_1\sigma_1,f_2\sigma_2)(y_j)|\\
&\le&C_3\mathcal{M}(f_1\sigma_1, f_2\sigma_2)(z_j)+|T(f_1^0\sigma_1,f_2^0\sigma_2)(y_j)|\\
&&\quad+|T(f_1\sigma_1\chi^{}_{Q_j^{**}\setminus \gamma Q_j}, f_2\sigma_2\chi^{}_{Q_j^{**}\setminus \gamma Q_j})(y_j)|\\
&&\quad+|T( f_{1}\sigma_{1}\chi_{Q_j^{**}\setminus \gamma Q_j},f_{2}^0\sigma_{2} )(y_j)|
+|T( f_{1}^0\sigma_{1},f_{2}\sigma_{2}\chi_{Q_j^{**}\setminus \gamma Q_j} )(y_j)|.
\end{eqnarray*}
Consequently, for any $0<\delta<1/2$,
\begin{eqnarray*}
&&|T(f_1\sigma_1 ,f_2\sigma_2)(x)-T(f_1^0\sigma_1, f_2^0\sigma_2)(x)-T(f_1\sigma_1,f_2\sigma_2)(y_j)|^\delta\\
&&\quad\le(C_3\beta)^\delta\lambda^\delta+\Sigma'^{\delta},
\end{eqnarray*}
where
\begin{eqnarray*}
\Sigma'&=&|T(f_1^0\sigma_1,f_2^0\sigma_2)(y_j)|+|T(f_1\sigma_1\chi^{}_{Q_j^{**}\setminus \gamma Q_j}, f_2\sigma_2\chi^{}_{Q_j^{**}\setminus \gamma Q_j})(y_j)|\\
&&\quad+|T( f_{1}\sigma_{1}\chi_{Q_j^{**}\setminus \gamma Q_j},f_{2}^0\sigma_{2} )(y_j)|
+|T( f_{1}^0\sigma_{1},f_{2}\sigma_{2}\chi_{Q_j^{**}\setminus \gamma Q_j} )(y_j)|.
\end{eqnarray*}
Note that $|Q_j^*\bigcap \Omega^c|$ is comparable with $|Q_j^{**}|$ due to the property of
the Whitney decomposition.
Integrating over $y_j\in Q_j^*\bigcap \Omega^c$, we have
\begin{eqnarray*}
&&\frac{1}{|Q_j^*\bigcap \Omega^c|}\int_{Q_j^*\bigcap \Omega^c}\Big|T(f_1\sigma_1 ,f_2\sigma_2)(x)-T(f_1^0\sigma_1, f_2^0\sigma_2)(x) \\
&& \hskip 9em -T(f_1\sigma_1,f_2\sigma_2)(y_j)\Big|^\delta dy_j\\
&\le&(C_3\beta\lambda)^\delta+\frac{1}{|Q_j^*\bigcap \Omega^c|}\int_{Q_j^*\bigcap \Omega^c}\Sigma'^\delta dy_j\\
&\le&(C_3\beta\lambda)^\delta+C_4\mathcal{M}(f_1\sigma_1,f_2\sigma_2)(z_j)^\delta,
\end{eqnarray*}
where Kolmogorov's inequality and the $L^1\times L^1\rightarrow L^{1/2,\infty}$ boundedness of $T$ are used, see \cite[p. 1239]{LOPTT}.
Since
\begin{eqnarray*}
&&|T(f_1\sigma_1 ,f_2\sigma_2)(x)-T(f_1^0\sigma_1, f_2^0\sigma_2)(x)-T(f_1\sigma_1,f_2\sigma_2)(y_j)|^\delta\\
&\ge& |T(f_1\sigma_1, f_2\sigma_2)(x)-T(f_1^0\sigma_1, f_2^0\sigma_2)(x)|^\delta-|T(f_1\sigma_1, f_2\sigma_2)(y_j)|^\delta\\
&\ge& |T(f_1\sigma_1, f_2\sigma_2)(x)-T(f_1^0\sigma_1, f_2^0\sigma_2)(x)|^\delta-\lambda^\delta,
\end{eqnarray*}
we have
\[
 |T(f_1\sigma_1, f_2\sigma_2)(x)-T(f_1^0\sigma_1, f_2^0\sigma_2)(x)|\le (1+C_5\beta)\lambda.
\]
It follows that  for $\beta\le (2C_5)^{-1}$,
\[
   |T(f_1\sigma_1, f_2\sigma_2)(x)-T(f_1^0\sigma_1, f_2^0\sigma_2)(x)|\le (1+C_5\beta)\lambda\le 3\lambda/2.
\]
Denote
\[
  E_j=\{x\in Q_j: |T(f_1\sigma_1, f_2\sigma_2)(x)|>2\lambda; \mathcal{M}(f_1\sigma_1 ,f_2\sigma_2)(x)\le \beta\lambda\}.
\]
Then we have
\[
  E_j\subset \{x\in Q_j: |T(f_1^0\sigma_1 ,f_2^0\sigma_2)(x)|>\lambda/2\}.
\]
Therefore,
\begin{eqnarray*}
\sum_j v_{\vec{w}}(E_j)
 \hskip -0.5em&\le& \beta^{-p} \hskip -0.3em
  \hskip -1em \sum_{j: v_{\vec{w}}(E_j)> \beta v_{\vec{w}}(\gamma Q_j)}\hskip -1em
 v_{\vec{w}}(E_j)\bigg( \frac{2}{\lambda v_{\vec{w}}(\gamma Q_j)}\int_{E_j}|T(f_1^0\sigma_1 ,f_2^0\sigma_2)|v_{\vec{w}}\bigg)^p\\
 &&\quad+ \beta\sum_{j: v_{\vec{w}}(E_j)\le \beta v_{\vec{w}}(\gamma Q_j)}v_{\vec{w}}(\gamma Q_j)\\
 &:=&I+II.
\end{eqnarray*}
By H\"older's inequality,  we have
\begin{eqnarray*}
I&\le& (\frac{2}{\beta\lambda})^p\mathcal{T}_*^p\sum_j\prod_{i=1}^2\left(\int_{\gamma Q_j} |f_i(y_i)|^{p_i}\sigma_i dy_i\right)^{p/{p_i}}\\
&\le&(\frac{2}{\beta\lambda})^p\mathcal{T}_*^p\prod_{i=1}^2\left(\int_{\bbR^n} (\sum_j\chi^{}_{\gamma Q_j})|f_i(y_i)|^{p_i}\sigma_i dy_i\right)^{p/{p_i}}\\
&\lesssim&\beta^{-p}\lambda^{-p}\mathcal{T}_*^p\prod_{i=1}^2\|f_i\|_{L^{p_i}(\sigma_i)}^p.
\end{eqnarray*}
On the other hand,
by the property of Whitney's decomposition,
\[
  II\le \beta \sum_j v_{\vec{w}}(\gamma Q_j)\le C_n\beta v_{\vec{w}}(\Omega).
\]
This proves $(\ref{eq:e1})$.
Taking the infimum over $\Omega$, we have
\begin{eqnarray*}
&&v_{\vec{w}}\{x\in \bbR^n: |T(f_1\sigma_1, f_2\sigma_2)(x)|>2\lambda; \mathcal{M}(f_1\sigma_1 ,f_2\sigma_2)(x)\le \beta \lambda\}\\
&\lesssim& \beta v_{\vec{w}}\{x\in \bbR^n: |T(f_1\sigma_1,f_2\sigma_2)(x)|>\lambda\}+\mathcal{T}_*^p\beta^{-p}\lambda^{-p}
\prod_{i=1}^2\|f_i\|_{L^{p_i}(\sigma_i)}^p.
\end{eqnarray*}
It follows that
\begin{eqnarray*}
&&\|T(f_1\sigma_1, f_2\sigma_2)\|_{L^{p,\infty}(v_{\vec{w}})}^p\\
&=&\sup_{\lambda>0}(2\lambda)^p v_{\vec{w}}\{|T(f_1\sigma_1 ,f_2\sigma_2)(x)|>2\lambda\}\\
&\le&\sup_{\lambda>0}(2\lambda)^p v_{\vec{w}}\{|T(f_1\sigma_1 ,f_2\sigma_2)(x)|>2\lambda; \mathcal{M}(f_1\sigma_1 ,f_2\sigma_2)(x)\le \beta \lambda\}\\
&&\quad+\sup_{\lambda>0}(2\lambda)^p v_{\vec{w}}\{\mathcal{M}(f_1\sigma_1, f_2\sigma_2)(x)> \beta \lambda\}\\
&\le&\sup_{\lambda>0}2^p\beta C_6 \lambda^pv_{\vec{w}}\{x\in \bbR^n: |T(f_1\sigma_1, f_2\sigma_2)(x)|>\lambda\}\\
&&\quad+(2^{p}\|\mathcal{M}\|_{L^{p_1}(w_1)\times  L^{p_2}(w_2)\rightarrow L^{p,\infty}(v_{\vec{w}})}^p+C_6\mathcal{T}_*^p)
\beta^{-p}\prod_{i=1}^2\|f_i\|_{L^{p_i}(\sigma_i)}^p\\
&=&2^p\beta C_6\|T(f_1\sigma_1, f_2\sigma_2)\|_{L^{p,\infty}(v_{\vec{w}})}^p\\
&&\quad+(2^{p}\|\mathcal{M}\|_{L^{p_1}(w_1)\times  L^{p_2}(w_2)\rightarrow L^{p,\infty}(v_{\vec{w}})}^p+C_6\mathcal{T}_*^p)
\beta^{-p}\prod_{i=1}^2\|f_i\|_{L^{p_i}(\sigma_i)}^p.
\end{eqnarray*}
Let $\beta=\min\{ (2C_5)^{-1}, (2^{p+1}C_6)^{-1} \}$. By Lemma~\ref{lm:mweak}, we get
\begin{equation}\label{eq:e13}
\|T(f_1\sigma_1,f_2\sigma_2)\|_{L^{p,\infty}(v_{\vec{w}})}\lesssim
(\mathcal{T}_*+[\vec{w}]_{A_{\vec{P}}}^{1/p})\prod_{i=1}^2\|f_i\|_{L^{p_i}(\sigma_i)}.
\end{equation}
\end{proof}

The following is a characterization of the weak boundedness of
$A_{\mathscr{D},\mathcal{S}}$.

\begin{Lemma}\label{lm:la}
Let $1<p,p_1,p_2<\infty$ and $\vec{w}\in A_{\vec{P}}$, where $\vec{w}:=(w_1,  w_2)$ and $\vec{P}:=(p_1,  p_2)$ with $1/p=1/{p_1}+1/{p_2}$.
Suppose that $\mathscr{D}$ is a dyadic grid and $\mathcal{S}$ is a sparse family in $\mathscr{D}$.
 Then the following assertions are equivalent.
\begin{enumerate}
\item $\|A_{\mathscr{D},\mathcal{S}}(|f_1|\sigma_1,|f_2|\sigma_2)\|_{L^{p,\infty}(v_{\vec{w}})}\le C\prod_{i=1}^2\|f_i\|_{L^{p_i}(\sigma_i)}$;
\item $\int_Q A_{\mathscr{D},\mathcal{S}}(|f_1|\sigma_1\chi^{}_Q,|f_2|\sigma_2\chi^{}_{Q})(x)v_{\vec{w}}(x)dx\le C\prod_{i=1}^2\|f_i \|_{L^{p_i}(\sigma_i)}v_{\vec{w}}(Q)^{1/{p'}}$ for all  cubes $Q\subset\bbR^n $
and all functions $f_i\in L^{p_i}(\sigma_i)$, $i=1,2$;
\item $\int_Q A_{\mathscr{D},\mathcal{S}}(|f_1|\sigma_1\chi^{}_Q,|f_2|\sigma_2\chi^{}_{Q})(x)v_{\vec{w}}(x)dx\le C\prod_{i=1}^2\|f_i \|_{L^{p_i}(\sigma_i)}v_{\vec{w}}(Q)^{1/{p'}}$ for all dyadic cubes $Q\in \mathcal{S}$
and all functions $f_i\in L^{p_i}(\sigma_i)$, $i=1,2$.
\end{enumerate}
\end{Lemma}
\begin{proof}
(i)$\Rightarrow$ (ii) is similar to the proof of Lemma~\ref{lm:weak}. (ii)$\Rightarrow$(iii) is obvious.
We only need to prove (iii)$\Rightarrow$ (i).

For any $t>0$, denote $\Omega_t:=\{x\in\bbR^n: A_{\mathscr{D},\mathcal{S}}(|f_1|\sigma_1,|f_2|\sigma_2)(x)>t\}:=\bigcup_{\zeta} P_\zeta$,
where $P_\zeta$ are pairwise disjoint maximal cubes in $\Omega_t$. We have
\[
  \sum_{Q_{j,k}\supset P_\zeta} \prod_{i=1}^2 \frac{1}{|Q_{j,k}|}\int_{Q_{j,k}}|f_i(y_i)|\sigma_idy_i >t,
\]
and
\[
  \sum_{Q_{j,k}\supsetneq P_\zeta} \prod_{i=1}^2 \frac{1}{|Q_{j,k}|}\int_{Q_{j,k}}|f_i(y_i)|\sigma_i dy_i \le t.
\]
Therefore, for $x\in E_\zeta=P_\zeta\cap \Omega_{2t}$, we have
\[
  2t<A^{}_{\mathscr{D},\mathcal{S}}(|f_1|\sigma_1 , |f_2|\sigma_2 )(x) \le A^{}_{\mathscr{D},\mathcal{S}}(|f_1|\sigma_1 1_{P_\zeta}, |f_2|\sigma_2 1_{P_\zeta})(x)+t.
\]
That is,
\[
  A^{}_{\mathscr{D},\mathcal{S}}(|f_1|\sigma_1 1_{P_\zeta}, |f_2|\sigma_2 1_{P_\zeta})(x)>t,\quad\mbox{$x\in E_\zeta$}.
\]
It follows that
\begin{eqnarray*}
&&(2t)^{p}v_{\vec{w}}(\Omega_{2t})\\
&=& (2t)^p\sum_\zeta v_{\vec{w}}(E_\zeta)\\
&\le& 2^p\!\! \sum_{v_{\vec{w}}(E_\zeta) > \beta v_{\vec{w}}(P_\zeta)}\!\!\!
  v_{\vec{w}}(E_\zeta)
\bigg(\frac{1}{v_{\vec{w}}(E_\zeta)}\int_{E_\zeta} A^{}_{\mathscr{D},\mathcal{S}}(|f_1|\sigma_1 1_{P_\zeta}, |f_2|\sigma_2 1_{P_\zeta})(x)dx\bigg)^p\\
&&\quad +(2t)^p\sum_{v_{\vec{w}}(E_\zeta) \le \beta v_{\vec{w}}(P_\zeta)}\beta v_{\vec{w}}(P_\zeta)\\
&\le&  2^p \beta^{1-p}\sum_{\zeta}v_{\vec{w}}(P_\zeta)^{1-p}
\bigg(\int_{P_\zeta} A^{}_{\mathscr{D},\mathcal{S}}(|f_1|\sigma_1 1_{P_\zeta}, |f_2|\sigma_2 1_{P_\zeta})(x)dx\bigg)^p\\
&&\quad +2^p\beta \|A^{}_{\mathscr{D},\mathcal{S}}(|f_1|\sigma_1 , |f_2|\sigma_2 )\|_{L^{p,\infty}(v_{\vec{w}})}^p\\
&\le&C^p 2^p \beta^{1-p}\sum_\zeta\bigg(\int_{P_\zeta} |f_1|^{p_1}\sigma_1\bigg)^{p/{p_1}}\cdot\bigg(\int_{P_\zeta} |f_2|^{p_2}\sigma_2\bigg)^{p/{p_2}}\\
&&\quad+2^p\beta \|A^{}_{\mathscr{D},\mathcal{S}}(|f_1|\sigma_1 , |f_2|\sigma_2 )\|_{L^{p,\infty}(v_{\vec{w}})}^p\\
&\le&C^p 2^p \beta^{1-p}\|f_1\|_{L^{p_1}(\sigma_1)}^p\|f_2\|_{L^{p_2}(\sigma_2)}^p+2^p\beta \|A^{}_{\mathscr{D},\mathcal{S}}(|f_1|\sigma_1 , |f_2|\sigma_2 )\|_{L^{p,\infty}(v_{\vec{w}})}^p.
\end{eqnarray*}
By setting $\beta=2^{-p-1}$ and taking
the supremum of $t$, we get the conclusion desired.
\end{proof}

Now we are ready to prove Theorem~\ref{thm:m1}.

\begin{proof}[Proof of Theorem~\ref{thm:m1}]
By setting the Banach space $\mathcal{X}$ to be $L^1_{v_{\vec{w}}}(Q)$, we see from
 (\ref{eq:x}) that
\begin{eqnarray*}
&&\int_Q |T(f_1\sigma_1\chi^{}_Q,f_2\sigma_2\chi^{}_{Q})(x)|v_{\vec{w}}(x)dx\\
&&\quad\lesssim
\sup_{\mathscr{D},\mathcal{S}}\int_Q A_{\mathscr{D},\mathcal{S}}(|f_1|\sigma_1\chi^{}_Q,|f_2|\sigma_2\chi^{}_{Q})(x)v_{\vec{w}}(x)dx.
\end{eqnarray*}
Hence
\begin{eqnarray*}
&&v_{\vec{w}}(Q)^{-1/{p'}}\int_Q |T(f_1\sigma_1\chi^{}_Q,f_2\sigma_2\chi^{}_{Q})(x)|v_{\vec{w}}(x)dx\\
&&\quad\lesssim
\sup_{\mathscr{D},\mathcal{S}}
v_{\vec{w}}(Q)^{-1/{p'}}\int_Q A_{\mathscr{D},\mathcal{S}}(|f_1|\sigma_1\chi^{}_Q,|f_2|\sigma_2\chi^{}_{Q})(x)v_{\vec{w}}(x)dx.
\end{eqnarray*}
For fixed $\mathscr{D}, \mathcal{S}$, by Lemma~\ref{lm:la}, it suffices to estimate
\[
 v_{\vec{w}}(Q)^{-1/{p'}}\int_Q A_{\mathscr{D},\mathcal{S}}(|f_1|\sigma_1\chi^{}_Q,|f_2|\sigma_2\chi^{}_{Q})(x)v_{\vec{w}}(x)dx
\]
for dyadic cube $Q\in\mathcal{S}$.
By  Lemma~\ref{lm:l2}, we have
\begin{eqnarray*}
&&v_{\vec{w}}(Q)^{-1/{p'}}\int_Q A_{\mathscr{D},\mathcal{S}}(|f_1|\sigma_1\chi^{}_Q,|f_2|\sigma_2\chi^{}_{Q})(x)v_{\vec{w}}(x)dx\\
&\le&v_{\vec{w}}(Q)^{-1/{p'}}
\int_Q A_{\mathscr{D},\mathcal{S}}(v_{\vec{w}}\chi^{}_Q, |f_2|\sigma_2\chi^{}_{Q})|f_1|\sigma_1 (x)dx\\
&\le& v_{\vec{w}}(Q)^{-1/{p'}}
\bigg(\int_Q (A_{\mathscr{D},\mathcal{S}}( v_{\vec{w}}\chi^{}_Q, |f_2|\sigma_2\chi^{}_{Q}))^{p_1'}\sigma_1 dx\bigg)^{1/{p_1'}}\\
&&\quad\cdot
\bigg(\int_Q |f_1|^{p_1}\sigma_1 (x)dx\bigg)^{1/{p_1}} \\
&\lesssim& [\vec{w}]_{A_{\vec{P}}}^{1/p}[v_{\vec{w}}]_{A_\infty}^{1/{p'}}([\sigma_1]_{A_\infty}^{1/{p_1}}+[\sigma_2]_{A_\infty}^{1/{p_2}})
\|f_1\|_{L^{p_1}(\sigma_1)}\|f_2\|_{L^{p_2}(\sigma_2)}.
\end{eqnarray*}
By (\ref{eq:e13}), we get the conclusion desired.
\end{proof}

\end{document}